\documentclass[12pt]{article}
\usepackage{amsmath,amssymb,amsthm,bbm,srcltx,graphicx,color,xargs,natbib}
\usepackage[dvipsnames]{xcolor}
\usepackage[shortlabels]{enumitem}

\usepackage[a4paper]{geometry}

\usepackage[colorlinks]{hyperref}

\numberwithin{equation}{section}

\newcommand{\dto}{\stackrel{d}{\longrightarrow}}
\newcommand{\Pto}{\stackrel{P}{\longrightarrow}}
\newcommand{\fidi}{\stackrel{\mathrm{fi.di.}}{\longrightarrow}}

\newcommand\1[1]{\mathbbm{1}_{#1}}
\newcommand\ind[1]{\mathbbm{1}{\left\{#1\right\}}}

\newcommand{\EE}{\mathbb{E}}
\newcommand{\NN}{\mathbb{N}}
\newcommand{\PP}{\mathbb{P}} 
\newcommand{\RR}{\mathbb{R}} 
\newcommand{\Rset}{\mathbb{R}} 
\newcommand{\ZZ}{\mathbb{Z}}

\newcommand{\law}{\mathcal{L}}

\newcommand{\bsQ}{\boldsymbol{Q}}
\newcommand{\bsX}{\boldsymbol{X}}
\newcommand{\bsY}{\boldsymbol{Y}}
\newcommand{\bsZ}{\boldsymbol{Z}}
\newcommand{\bTheta}{\boldsymbol{\Theta}}

\newcommand{\by}{\boldsymbol{y}}
\newcommand{\bx}{\boldsymbol{x}}
\newcommand{\tbx}{\tilde{\boldsymbol{x}}}
\newcommand{\tby}{\tilde{\boldsymbol{y}}}

\newcommand\dtilde{\tilde{d}}

\newcommand\mca{\mathcal{A}}
\newcommand\mcb{\mathcal{B}}
\newcommand\mcm{\mathcal{M}}

\newcommand\rmd{\mathrm{d}}
\newcommand\rme{\mathrm{e}}
\newcommand\rmi{\mathrm{i}}

\newcommand\location{a}
\newcommand\skewness{\phi}
\newcommand\intrv{K}
\newcommand\pointmeas{\gamma}
\newcommand\paretorec{\zeta}
\newcommand\hausdorff{m}
\newcommand\metricE{m_E}
\newcommand\unifE{m^*}

\newcommand\shift{B}

\newcommand{\lo}{\tilde{l}_0}

\newcommand\PPC{N''}

\newcommand\setJ{J} 

\newcommandx\sequence[3][2=\ZZ,3=t]{\{#1_#3,#3\in#2\}}
\newcommandx\sequenceshort[2][2=t]{\{#1_#2\}}

\newcommand{\var}{\operatorname{var}}

\newcommand\cadlag{c\`adl\`ag}
\newcommand\eg{e.g.}
\newcommand\ie{i.e.}

\newcommand\iid{i.i.d.}

\usepackage{cleveref}

\newtheorem{theorem}{Theorem}[section]
\newtheorem{lemma}[theorem]{Lemma}

\newtheorem{proposition}[theorem]{Proposition}
\newtheorem{hypothesis}[theorem]{Assumption}
\theoremstyle{remark}
\newtheorem{remark}[theorem]{Remark}
\newtheorem{example}[theorem]{Example}

\crefname{hypothesis}{Assumption}{Assumptions}

\begin{document}

\title{An invariance principle for sums and record times of regularly varying stationary sequences}

\author{Bojan Basrak\thanks{Department of Mathematics, University of Zagreb, Bijeni\v cka 30,
    Zagreb, Croatia} \and Hrvoje Planini\'c\thanks{Department of Mathematics, University of Zagreb,
    Bijeni\v cka 30, Zagreb, Croatia} \and Philippe Soulier\thanks{Laboratoire MODAL'X, 'epartement de Math\'ematique
    et Informatique, Universit\'e Paris Nanterre, 92000 Nanterre, France}}

\maketitle

\begin{abstract}
  We prove a sequence of limiting results about weakly dependent stationary and regularly varying
  stochastic processes in discrete time. After deducing the limiting distribution for individual
  clusters of extremes, we present a new type of point process convergence theorem.  It is designed
  to preserve the entire information about the temporal ordering of observations which is typically
  lost in the limit after time scaling.  By going beyond the existing asymptotic theory, we are able
  to prove a new functional limit theorem. Its assumptions are satisfied by a wide class of applied
  time series models, for which standard limiting theory in the space $D$ of \cadlag\ functions does
  not apply.
  
  To describe the limit of partial sums in this more general setting, we use the space~$E$ of
  so--called decorated \cadlag\ functions. We also study the running maximum of partial sums for
  which a corresponding functional theorem can be still expressed in the familiar setting of space
  $D$.
  
  We further apply our method to analyze record times in a sequence of dependent stationary
  observations, even when their marginal distribution is not necessarily regularly varying.  Under
  certain restrictions on dependence among the observations, we show that the record times after
  scaling converge to a relatively simple compound scale invariant Poisson process.

\end{abstract}

\section{Introduction}

Donsker--type functional limit theorems represent one of the key developments in probability theory.
They express invariance principles for rescaled random walks of the form
\begin{equation} 
  \label{e:RW}
  S_{\lfloor nt  \rfloor} = X_1 + \cdots + X_{\lfloor nt  \rfloor}\,,\quad t\in [0,1]\,.
\end{equation} 
Many extension of the original invariance principle exist, most notably allowing dependence between
the steps $X_i$, or showing, like Skorohod did, that non--Gaussian limits are possible if the steps
$X_i$ have infinite variance. For a survey of invariance principles in the case of dependent
variables in the domain of attraction of the Gaussian law, we refer to
\cite{merlevede:peligrad:utev:2006}, see also \cite{bradley:book} for a thorough survey of mixing
conditions.  In the case of a non--Gaussian limit, the limit of the processes $(S_{\lfloor nt
  \rfloor})_{t\in[0,1]}$ is not a continuous process in general.  Hence, the limiting theorems of
this type are placed in the space of \cadlag\ functions denoted by $D([0,1])$ under one of the
Skorohod topologies. The topology denoted by $J_1$ is the most widely used (often implicitely) and
suitable for \iid\ steps, but over the years many theorems involving dependent steps have been shown
using other Skorohod topologies. Even in the case of a simple $m$--dependent linear process from a
regularly varying distribution, it is known that the limiting theorem cannot be shown in the
standard $J_1$ topology, see Avram and Taqqu \cite{avram:taqqu:1992}.  Moreover, there are examples
of such processes for which none of the Skorohod topologies work, see \Cref{sec:consum}.

However, as we found out, for all those processes and many other stochastic models relevant in
applications, random walks do converge, but their limit exists in an entirely different space. To
describe the elements of such a space we use the concept of decorated \cadlag\ functions and denote
 the corresponding space by $E([0,1])$, following Whitt \cite{whitt:2002}. See \Cref{sec:consum}. Presentation of this new
type of limit theorem is the main goal of our article. For the statement of our main result see
\Cref{thm:PartialSumConvinE} in \Cref{sec:consum}.  As a related goal we also study the running
maximum of the random walk~$S_{\lfloor nt \rfloor}$ for which, due to monotonicity, the limiting
theorem can still be expressed in the familiar space $D([0,1])$.

Our main analytical tool is the limit theory for point processes in a certain nonlocally compact
space which is designed to preserve the order of the observations as we rescale time to interval
$[0,1]$ as in \eqref{e:RW}.  Observe that due to this scaling, successive observations collapse in
the limit to the same time instance. As the first result in this context, we prove in
Section~\ref{sec:clusters} a limit theorem related to large deviations results of the type shown
recently by \cite{mikosch:wintenberger:2016} (cf. also \cite{hult:samorodnitsky:2010}) and offer an
alternative probabilistic interpretation of these results.  Using our setup, we can group successive
observations in the sequence $\{X_i,\, i =1,\ldots, n\}$, in nonoverlapping clusters of increasing
size to define a point process which completely preserves the information about the order among the
observations. This allows us to show in a rather straightforward manner that so constructed
empirical point processes converge in distribution towards a Poisson point process on an appropriate
state space.  The corresponding theorem could be arguably considered as the key result of the
paper. It motivates all the theorems in the later sections and extends point process limit theorems
in \cite{davis:hsing:1995} and \cite{basrak:krizmanic:segers:2012}, see Section \ref{sec:poipro}.
 
Additionally, our method allows for the analysis of records and record times in the sequence of dependent
stationary observations $X_i$. By a classical result of R\'enyi, the number of records among first
$n$ iid observations from a continuous distribution grows logarithmically with $n$. Moreover it is
known (see \eg\ \cite{resnick:1987}) that record times rescaled by $n$ tend to the so--called scale
invariant Poisson process, which plays a fundamental role in several areas of probability, see
\cite{arratia:1998}.  For a stationary sequence with an arbitrary continuous
marginal distribution, we show that the record times converge to a relatively simple compound
Poisson process under certain restrictions on dependence. This form of the limit reflects the fact
that for dependent sequences records tend to come in clusters, as one typically observes in many
natural situations.  This is the content of Section \ref{sec:records}. Finally, proofs of certain
technical auxiliary results are postponed to \Cref{sec:proof-trivia}.  In the rest of the
introduction, we formally present the main ingredients of our model.

We now introduce our main assumptions and notation. Let $\|\cdot\|$ denote an arbitrary norm on
$\RR^d$ and let $\mathbb{S}^{d-1}$ be the corresponding unit sphere.  Recall that a $d$-dimensional
random vector $\bsX$ is regularly varying with index $\alpha > 0$ if there exists a random vector
$\boldsymbol{\Theta} \in \mathbb{S}^{d-1}$ such that
\begin{equation}
  \label{def:RV_d}
  \frac{1}{\PP(\|\bsX\| > x)} \PP(\|\bsX\| > ux, \bsX / \|\bsX\| \in \cdot )\Rightarrow u^{-\alpha} \PP(\bTheta \in \cdot),
\end{equation}
for every $u>0$ as $x \to\infty$, where $\Rightarrow$ denotes the weak convergence of measures, here on
$\mathbb{S}^{d-1}$.  An $\RR^d$-valued time series $\sequence{X}$ is regularly varying if
all the finite-dimensional vectors $(X_k,\dots,X_l)$, $k\leq l \in \ZZ$ are regularly varying, see
\cite{davis:hsing:1995} for instance. We will consider a stationary regularly varying process
$\sequence{X}$. The regular variation of the marginal distribution implies that
there exists a sequence $\{a_n\}$ which for all $x>0$ satisfies
\begin{align}
  \label{eq:a}
  n\PP(\|X_0\| > a_n x)\to x^{-\alpha} \; .
\end{align} 
  If $d=1$, it is known that 
  \begin{equation}
  \label{eq:conv-mu}
  n\PP(X_0 /a_n \in \cdot) \stackrel{v}{\longrightarrow} \mu  \; , 
\end{equation}
where $\stackrel{v}{\longrightarrow}$ denotes vague convergence on $\RR\setminus \{0\}$ with the measure $\mu$ on $\RR \setminus\{0\}$
given by
\begin{equation}
  \label{eq:def-mu}
  \mu(dy)=p\alpha y^{-\alpha-1} \1{(0,\infty)}(y) \rmd y + (1-p)\alpha (-y)^{-\alpha-1} \1{(-\infty,0)}(y) \rmd y
\end{equation}
for some $p\in[0,1]$. 

According to \cite{basrak:segers:2009}, the regular variation of the stationary sequence
$\sequence{X}$ is equivalent to the existence of an $\RR^d$ valued time series $\sequence{Y}$
called the tail process which satisfies $\PP(\|Y_0\| > y)=y^{-\alpha}$ for $y \geq 1$ and, as
$x\to \infty$,
\begin{align}
  \label{eq:tailprocess}
  \left ( \sequence{x^{-1}X} \, \big| \, \|X_0\| > x \right)  \fidi \sequence{Y} \; ,
\end{align}
where $\fidi$ denotes convergence of finite-dimensional distributions. Moreover, the so-called
spectral tail process $\sequence{\Theta}$, defined by $\Theta_t = Y_t/\|Y_0\|$, $t \in \ZZ$, 
turns out to be independent of $\|Y_0\|$ and satisfies
\begin{equation}
  \label{eq:theta}
  \left( \sequence{|X_0|^{-1} X} \, \big| \, \|X_0\| > x \right) \fidi \sequence{\Theta} \; , 
\end{equation}
as $x\to\infty$. If $d=1$, it follows that $p$ from (\ref{eq:def-mu}) satisfies $p=\PP(\Theta_0=1)=1-\PP(\Theta_0=-1)$.

We will often assume in addition that the following condition, referred to as the
anticlustering or finite mean cluster length condition, holds.
\begin{hypothesis}
\label{hypo:AC}
There exists a sequence of integers $(r_n)_{n\in\NN}$ such that $\lim_{n\to\infty} r_n
=\lim_{n\to\infty} n/r_n =\infty$ and  for every $u > 0$,
\begin{align}
  \label{eq:AC}
  \lim_{m \to \infty} \limsup_{n \to \infty} \PP \biggl( \max_{m \le |i| \le r_{n}} \|X_{i}\| > a_n  u\,\bigg|\,\|X_{0}\|>a_n u \biggr) = 0 \; .
\end{align}
\end{hypothesis}

There are many time series satisfying the conditions above including several nonlinear models like
stochastic volatility or GARCH (see \cite[Section~4.4]{mikosch:wintenberger:2013}). 

In the sequel, an important role will be played by the quantity $\theta$ defined by
\begin{align}
    \label{eq:def-theta}
    \theta = \PP\left(\sup_{t\geq 1} \|Y_t\| \leq 1\right) \; .
  \end{align}
It was shown in \cite[Proposition 4.2]{basrak:segers:2009} that \Cref{hypo:AC} implies that $\theta>0$.

\section{Asymptotics of clusters}
\label{sec:clusters} 
Let $l_0$ be the space of double-sided $\RR^d$-valued sequences converging to zero at both
ends, \ie~$l_0 = \{\bx= \sequence{x}[\ZZ][i] : \lim_{|i|\to\infty}\|x_i\|= 0\}$.  On $l_0$ consider the
uniform norm
\begin{align*}
  \|\bx\|_\infty = \sup_{i\in\ZZ} \|x_i\| \; ,
\end{align*}
which makes $l_0$ into a separable Banach space. Indeed, $l_0$ is the closure of all double-sided
rational sequences with finitely many non zero terms in the Banach space of all bounded double-sided
real sequences.  Define the shift operator $\shift$ on $l_0$ by $(\shift\bx)_i = x_{i+1}$ and
introduce an equivalence relation $\sim$ on $l_0$ by letting $\bx \sim \by $ if $\by=\shift^k\bx$
for some $k\in\ZZ$.  In the sequel, we consider the quotient space
$$
\lo = l_0/\sim \; ,
$$
and define a function $\tilde{d}:\lo\times\lo\longrightarrow [0,\infty)$ by
\begin{align*}
  \tilde{d}(\tilde{\bx},\tilde{\by}) =
  \inf\{\|\bx'-\by'\|_\infty:\bx'\in\tilde{\bx},\by'\in\tilde{\by}\} =
  \inf\{\|\shift^k\bx-\shift^l\by\|_\infty:k,l\in\ZZ\} \; .
\end{align*}
for all $\tilde{\bx},\tilde{\by}\in \lo$, and all $\bx\in\tilde{\bx},\by\in\tilde{\by}$.  The proof
of the following result can be found in \Cref{sec:proof-trivia}.
\begin{lemma}
  \label{lem:lo-complete}
  The function $\tilde{d}$ is a metric which makes $\lo$ a separable and complete
  metric space.
\end{lemma}
One can naturally embed the set $\cup_{d\geq1} (\RR^d)^n \cup l_0$ into $\lo$ by mapping $\bx \in l_0$
to its equivalence class and an arbitrary finite sequence $\bx=(x_1,\ldots, x_n) \in (\RR^d)^n$ to the
equivalence class of the sequence
\[
(\ldots,0,0,\bx,0,0,\ldots) \; ,
\]
which adds zeros in front and after it.

Let $\sequence{Z}$ be a sequence distributed as the tail process $\sequence{Y}$ conditionally on the
event $\{\sup_{i\leq -1} \|Y_i\| \leq 1\}$ which, under \Cref{hypo:AC}, has a strictly positive
probability (cf. \cite[Proposition~4.2]{basrak:segers:2009}). More precisely,
\begin{equation}
  \label{eq:z}
  \mathcal{L}\left( \sequence{Z}\right) = \mathcal{L}\left( \sequence{Y} \, \Big|\, \sup_{i \leq -1} \|Y_i \| \leq 1 \right)\,.
\end{equation}
Since $(\ref{eq:AC})$ implies that  $\PP(\lim_{|t|\to\infty}\|Y_t\|= 0)=1$, see \cite[Proposition~4.2]{basrak:segers:2009}), the sequence $\{Z_t\}$ in \eqref{eq:z} can be viewed
as a random element in $l_0$ and $\lo$ in a natural way.  In particular, the random variable
\begin{equation*}
  L_Z=\sup_{j\in \ZZ} \|Z_j\| \; , 
\end{equation*} 
is a.s. finite and not smaller than 1 since $\PP(\|Y_0\| > 1)=1$.  Due to regular variation and (\ref{eq:AC}) one can show (see
\cite{basrak:tafro:2015}) that for $v\geq 1$
\begin{align*}
  \PP(L_Z>v)= v^{-\alpha} \; .
\end{align*}
One can also define a new sequence $\sequence{Q}$ in $\lo$ as the equivalence class of
\begin{equation} 
  \label{eq:q}
  Q_t = Z_t / L_Z \; , \ t\in \ZZ \; .
\end{equation}

Consider now a block of observations $( X_1,\ldots, X_{r_n})$ and define
$M_{r_n} = \max_{1\leq i \leq r_n} \|X_i\|$. It turns out that conditionally on the event that
$M_{r_n} > a_n u$, the law of such a block has a limiting distribution and that $L_Z$ and $\{Q_t\}$
are independent.
\begin{theorem}
  \label{lem:conv-cluster}
  Under \Cref{hypo:AC}, for every $u >0$,
 \begin{align*}
   \law \left( \frac{X_1,\ldots, X_{r_n}}{a_nu} \, \Big|\, M_{r_n} > a_nu\right) \dto \law
   \left(\sequence{Z} \right) \; , 
 \end{align*}
 as $n \to\infty$ in $\lo$. Moreover, $\sequence{Q}$ and $L_Z$ in \eqref{eq:q} are independent
 random elements with values in $\lo$ and $[0,\infty)$ respectively.
\end{theorem}

\begin{proof}
  \begin{enumerate}[{Step} 1,wide=0pt]
  \item We write $\bsX_n(i,j) =(X_{i},\dots,X_j)/a_nu$, $\bsY(i,j)=(Y_{i},\dots,Y_{j})$ and
    $M_{k,l}=\max_{k\leq i\leq l} \|X_i\|$, $M_{k,l}^Y=\max_{k\leq i\leq l}\|Y_i\|$. By the Portmanteau
    theorem \cite[Theorem~2.1]{billingsley:1968}, it suffices to prove that
  \begin{align}
    \lim_{n\to\infty} \EE[g(\bsX_n(1,r_n)) \mid M_{1,r_n} > a_nu ] & = \EE [g(\sequence{Y})
    \mid M_{-\infty,-1}^Y \leq 1 ] \; ,\label{eq:portmanteau}
  \end{align}
  for every nonnegative, bounded and uniformly continuous function $g$ on $(\lo,\dtilde)$.  

  Define the truncation $\tbx_\zeta$ at level $\zeta$ of $\tbx\in\lo$ by putting all the coordinates
  of $\tbx$ which are no greater in absolute value than $\zeta$ at zero, that is $\tbx_\zeta$ is the
  equivalence class of $(x_i\1{|x_i|>\zeta})_{i\in\ZZ}$, where $\bx$ is a representative of
  $\tbx$. Note that by definition, $\dtilde(\tbx,\tbx_\zeta)\leq\zeta$.

  For a function $g$ on $\lo$, define $g_\zeta$ by $g_\zeta(\tbx)=g(\tbx_\zeta)$.  If $g$ is
  uniformly continuous, then for each $\eta>0$, there exists $\zeta$ such that $|g(\tbx)-g(\tby)|\leq
  \eta$ if $\dtilde(\tbx,\tby)\leq \zeta$, that is,  $\|g-g_\zeta\|_\infty\leq\eta$.  Thus it is
  sufficent to prove~(\ref{eq:portmanteau}) for $g_\zeta$ for   $\zeta \in(0,1)$.

  One can now follow the steps of the proof of \cite[Theorem 4.3]{basrak:segers:2009}.  Decompose
  the event $\{M_{1,r_n} > a_nu\}$ according to the smallest $j\in\{1,\dots,r_n\}$ such that
  $\|X_j\|>a_nu.$ We have
  \begin{align}
    \EE[g_\zeta(\bsX_n(1,r_n))  ; M_{1,r_n} > a_nu ]=\sum_{j=1}^{r_n} \EE\left[g_\zeta(\bsX_n(1,r_n))
    ; M_{1,j-1}\leq a_nu<\|X_j\| \right] \; .  \label{eq:decomp-firstexceed}
  \end{align}
  Fix a positive integer $m$ and let $n$ be large enough so
  that $r_n\geq 2m+1.$ By the definition of $g_\zeta$, for all $j\in\{m+1,\dots,r_n-m\}$ we have
  that
   \begin{align}
     M_{1,j-m-1} \vee M_{j+m+1,r_n} \leq a_nu \zeta \Rightarrow g_\zeta(\bsX_n(1,r_n)) =
     g_\zeta(\bsX_n(j-m,j+m)) \; . \label{eq:equal-exceptif}
   \end{align}
   The proof is now exactly along the same lines as the proof of \cite[Theorem
   4.3]{basrak:segers:2009} and we omit some details.  Using stationarity, the
   decomposition~(\ref{eq:decomp-firstexceed}), the relation~(\ref{eq:equal-exceptif}) and the
   boundedness of $g$, we have,
   \begin{align}
     &     \left| \EE[g_\zeta(\bsX_n(1,r_n)) \ind{M_{1,r_n} > a_nu} ] \right.  \nonumber \\
     &  \left.  - r_n\EE\left[g_\zeta(\bsX_n(-m,m)) \ind{M_{-m,-1} \leq a_nu}\ind{\|X_0\| > a_nu} \right] \right| \nonumber \\
     &   \leq 2m \|g\|_\infty \PP(\|X_0\|>a_nu) + r_n\|g\|_\infty \PP (M_{-r_n,-m-1} \vee M_{m+1,r_n} >
       a_nu; \|X_0\|>a_nu) \; . \label{eq:argument}
   \end{align}
   Next define $\theta_n = \PP(M_{1,r_n}>a_nu) / \{r_n\PP(\|X_0\|>a_nu)\}$.  Under
   Assumption~\ref{hypo:AC}, 
   \begin{align}
     \lim_{n\to\infty} \theta_n = \PP(\sup_{i\geq1} \|Y_i\|\leq1) = \theta \; , 
     \label{eq:def-candidate-extremalindex}
   \end{align}
   where $\theta$ was defined in (\ref{eq:def-theta}). See \cite[Proposition~4.2]{basrak:segers:2009}.  Therefore, by \Cref{hypo:AC},
   (\ref{eq:argument}) and (\ref{eq:def-candidate-extremalindex}) we conclude that
   \begin{multline}
     \lim_{m\to\infty}\limsup_{n\to\infty}\bigg|\EE[g_\zeta(\bsX_n(1,r_n))  \mid M_{1,r_n} > a_nu ]\\
     -\frac{1}{\theta_n}\EE[g_\zeta(\bsX_n(-m,m)) ; M_{-m,-1} \leq a_nu \mid \|X_0\|>a_nu ]
     \bigg|=0\;. \label{eq:tightness-1}
   \end{multline}
   We now argue that, for every $m\geq 1$,
   \begin{multline}
     \lim_{n\to\infty} \EE [g_\zeta(\bsX_n(-m,m)); M_{-m,-1} \leq a_nu \mid \|X_0\| > a_nu]  \\
     = \EE [g_\zeta(\bsY(-m,m)); M_{-m,-1}^Y \leq 1] \; . \label{eq:argue}
   \end{multline}
   First observe that $g_\zeta$, as a function on $(\RR^d)^{2m+1}$, is continuous except maybe on the
   set
   $D_\zeta^{2m+1}=\{(x_1,\dots,x_{2m+1})\in(\RR^d)^{2m+1};|x_i|=\zeta \textnormal{ for some }
   i\in\{1,\dots, 2m+1\}\}$
   and we have that $\PP(\bsY(-m,m)\in D_\zeta^{2m+1})=0$ since $\bsY=\|Y_0\|\boldsymbol\Theta$,
   $\|Y_0\|$ and $\boldsymbol\Theta$ are independent and the distribution of $\|Y_0\|$ is Pareto
   therefore atomless. Observe similarly that the distribution of
   $M_{k,l}^{Y}=\|Y_0\|\max_{k\leq j\leq l} \|\Theta_j\|$ does not have atoms except maybe at
   zero. Therefore, since $g_\zeta$ is bounded, (\ref{eq:argue}) follows by the definition of the
   tail process and the continuous mapping theorem.
  
   Finally, since $\bsY(-m,m)\longrightarrow \bsY$ a.s. in $\lo$ and since $\bsY$ has only finitely
   many coordinates greater than $\zeta$, $g_\zeta(\bsY(-m,m))=g_\zeta(\bsY)$ for large enough $m$,
   almost surely. 
   Thus, applying~(\ref{eq:tightness-1}) and~(\ref{eq:argue}), we obtain by bounded convergence
   \begin{align*}
     \lim_{n\to\infty} 
     & \EE[g_\zeta(\bsX_n(1,r_n)) \mid M_{1,r_n} > a_nu] \\
     & = \lim_{m\to\infty}\lim_{n\to\infty}\frac{1}{\theta_n} \EE [g_\zeta(\bsX_n(-m,m)); M_{-m,-1}
       \leq a_nu \mid \|X_0\| > a_nu] \\
     & =   \frac{1}{\theta} \lim_{m\to\infty} \EE [g_\zeta(\bsY(-m,m)); M_{-m,-1}^Y\leq1 ]     \\
     & = \frac{1}{\theta} \EE [g_\zeta(\bsY); M_{-\infty,-1}^Y \leq 1] \; .
   \end{align*}
   Applying this to $g\equiv1$ we obtain $\theta=\PP(M_{-\infty,-1}^Y \leq 1)$. Hence
   (\ref{eq:portmanteau}) holds for $g_\zeta$ as we wanted to show.

 \item Observing that the mapping $\tilde{x} \mapsto(\tilde{x}, \|\bx\|_\infty )$ is continuous on
   $\lo$, we obtain for every~$u>0$,
  \begin{equation}
    \label{e:main_pomocni_1}
    \law \left( \frac{X_1,\ldots, X_{r_n}}{a_nu}\,, \, 
      \frac{M_{r_n}}{a_nu} \Big|\, M_{r_n} > a_nu\right) \dto \law
    \left(\{Z_t\} \,, \,  L_Z \right) \; .
  \end{equation}
  Similarly, the mapping defined on $\lo\times(0,\infty)$ by $(\tilde{x},b) \mapsto \tilde{x}/b$ is again continuous.  Hence,
  \eqref{e:main_pomocni_1} implies
  \begin{equation}
    \label{e:main_pomocni_2}
    \law \left( \frac{X_1,\ldots, X_{r_n}}{M_{r_n}}\,, \, 
      \frac{M_{r_n}}{a_nu} \Big|\, M_{r_n} > a_nu\right) \dto \law
    \left(\{Q_t\} \,, \,  L_Z \right)
  \end{equation}
  by the continuous mapping theorem.  To show the independence between $L_Z$ and $\{Q_t\}$, it
  suffices to show
  \begin{equation}
    \label{e:nez1}
    \EE \left[ g\left(\{Q_t\}\right) \mathbbm{1}_{\{L_Z >v\}} \right]=\EE \left[   g\left(\{Q_t\}\right) \right]  P({L_Z >v })\,,
  \end{equation}
  for an arbitrary uniformly continuous function $g$ on $\lo$ and $v \geq 1$.

  By \eqref{e:main_pomocni_2}, the left-hand side of \eqref{e:nez1} is the limit of
 $$\EE \left[ g \left(\frac{X_1,\ldots, X_{r_n}}{M_{r_n}} \right)\,\mathbbm{1}_{\{(a_n)^{-1}M_{r_n}>v\}}\bigg\vert M_{r_n}>a_n\right],$$
 which further equals
  $$\EE  \left[g \left(\frac{X_1,\ldots, X_{r_n}}{M_{r_n}} \right)
    \, \bigg\vert M_{r_n}>a_n v\right]\frac{\PP(M_{r_n}>a_n v)}{\PP(M_{r_n}>a_n )}.$$
  By \eqref{e:main_pomocni_2}, the first term in the product above tends to
  $\EE\left[ g\left(\{Q_t\}\right) \right]$ as $n\to\infty$. On the other hand, by
  (\ref{eq:def-candidate-extremalindex}) and regular variation of $\|X_0\|$, the second term tends
  to $v^{-\alpha}=\PP(L_Z>v)$.
\end{enumerate}
\end{proof}

\section{The point process of clusters}
\label{sec:poipro}
In this section we prove our main result on the point process asymptotics
for the sequence $\sequence{X}$. Prior to that, we discuss the topology of $w^\#$-convergence. 
\subsection{Preliminaries on $w^\#$-convergence}
\label{sec:preliminarues-weakstar}
To study convergence in distribution of point processes on the non locally-compact space $\lo$ we
use $w^{\#}$-convergence and refer to \cite[Section
A2.6.]{daley:verejones:2008t1} and \cite[Section 11.1.]{daley:verejones:2008t2} for details.  Let
$\mathbb{X}$ be a complete and separable metric space and let $\mathcal{M}(\mathbb{X})$ denote the
space of boundedly finite nonnegative Borel measures $\mu$ on $\mathbb{X}$, i.e. such that
$\mu(B)<\infty$ for all bounded Borel sets $B$. The subset of $\mathcal{M}(\mathbb{X})$ of all point
measures (that is measures $\mu$ such that $\mu(B)$ is a nonnegative integer for all bounded Borel sets $B$) is
denoted by $\mathcal{M}_p(\mathbb{X}).$ A sequence $\{\mu_n\}$ in $\mathcal{M}(\mathbb{X})$ is said
to converge to $\mu$ in the $w^\#$-topology, noted by $\mu_n\rightarrow_{w^\#}\mu$, if
\[
\mu_n(f)=\int f d\mu_n \rightarrow \int f d\mu=\mu(f),
\]
for every bounded and continuous function $f:\mathbb{X}\rightarrow \RR$ with bounded
support. Equivalently (\cite[Proposition A2.6.II.]{daley:verejones:2008t1}),
$\mu_n\rightarrow_{w^\#}\mu$ refers to
\[
\mu_n(B)\rightarrow\mu(B)
\] 
for every bounded Borel set $B$ with $\mu(\partial B)=0.$ We note that when $\mathbb{X}$ is
  locally compact, an equivalent metric can be chosen in which a set is relatively compact if and
  only if it is bounded, and $w^\#$-convergence coincides with vague convergence. We refer to
  \cite{kallenberg:2017} or \cite{resnick:1987} for details on vague convergence. The notion
of $w^\#$-convergence is metrizable in such a way that $\mathcal{M}(\mathbb{X})$ is Polish
(\cite[Theorem A2.6.III.(i)]{daley:verejones:2008t1}). Denote by
$\mathcal{B}(\mathcal{M}(\mathbb{X}))$ the corresponding Borel sigma-field.

It is known, see \cite[Theorem 11.1.VII]{daley:verejones:2008t2}, that a sequence $\{N_n\}$ of random
elements in $(\mathcal{M}(\mathbb{X}),\mathcal{B}(\mathcal{M}(\mathbb{X})))$, converges in
distribution to $N$, denoted by $N_n\dto N$, if and only if
\[
(N_n(A_1),\dots,N_n(A_k))\dto (N(A_1),\dots,N(A_k)) \textnormal{  in } \RR^k,
\]
for all $k\in\NN$ and all bounded Borel sets $A_1,\dots,A_k$ in $\mathbb{X}$ such that $N(\partial
A_i)=0$ a.s. for all $i=1,\dots,k$.

\begin{remark}
  \label{rem:LaplaceSmallerFamily}
  As shown in \cite[Proposition~11.1.VIII]{daley:verejones:2008t2}, this is equivalent to the
  pointwise convergence of the Laplace functionals, that is,
  $\lim_{n\to\infty}\EE[e^{-N_n(f)}]=\EE[e^{-N(f)}]$ for all bounded and continuous function $f$ on
  $\mathbb{X}$ with bounded support.  It turns out that it is sufficient (and more convenient in our
  context) to verify the convergence of Laplace functionals for a smaller convergence determining  family. 
  See the comments before \Cref{hypo:Aprimecluster}.
\end{remark}

\subsection{Point process convergence} 
\label{Subs:ppc} 
Consider now the space $\lo\setminus\{\tilde{\boldsymbol{0}}\}$ with the subspace
topology. Following  \cite{kallenberg:2017}, we metrize the space
$\lo\setminus\{\tilde{\boldsymbol{0}}\}$ with the complete metric
\[
\tilde{d}'(\tilde{\bx},\tilde{\by})=\left(\tilde{d}(\tilde{\bx},\tilde{\by})\wedge 1\right) \vee
\left|1/\|\tilde{\bx}\|_\infty - 1/\|\tilde{\by}\|_\infty\right|
\]
which is topologically equivalent to $\tilde{d}$, i.e.\ it generates the same (separable) topology
on $\lo\setminus\{\tilde{\boldsymbol{0}}\}$. However, a subset $A$ of
$\lo\setminus\{\tilde{\boldsymbol{0}}\}$ is bounded for $\tilde{d}'$ if and only if there exists
an $\epsilon>0$ such that $\tilde{x}\in A$ implies $\|\tilde{\bx}\|_\infty>\epsilon$. Therefore, for
measures $\mu_n,\mu\in\mathcal{M}(\lo\setminus\{\tilde{\boldsymbol{0}}\})$,
$\mu_n\rightarrow_{w^\#}\mu$ if $\mu_n(f)\rightarrow\mu(f)$ for every bounded and continuous
function $f$ on $\lo\setminus\{\tilde{\boldsymbol{0}}\}$ such that for some $\epsilon>0$,
$\|\tilde{\bx}\|_\infty\leq \epsilon$ implies $f(\tilde{\bx})=0$.
\begin{remark} 
  We note that under the metric $\tilde{d}'$, $w^\#$-convergence coincides with the theory of $M_0$-convergence introduced
  in \cite{hult:lindskog:2006}, further developed in \cite{lindskog2014regularly} and with the
  corresponding point process convergence recently studied by \cite{zhao2016point}.
\end{remark}

Take now a sequence $\{r_n\}$ as in \Cref{hypo:AC}, set $k_{n} = \lfloor n / r_{n} \rfloor$ and define
\[
\bsX_{n,i} =(X_{(i-1)r_n+1},\dots,X_{ir_n})/a_n
\] 
for $i=1,\dots,k_n.$ As the main result of this section we show, under certain conditions, the point
process of clusters $\PPC_n$ defined by
\[
\PPC_n = \sum_{i=1}^{k_n} \delta_{(i/k_n,\bsX_{n,i})}  \; 
\] 
restricted to
$[0,1] \times \lo\setminus\{\tilde{\boldsymbol{0}}\}$ (i.e.\ we ignore indices $i$ with  $\bsX_{n,i}=0$), converges in distribution in
$\mathcal{M}_p([0,1] \times \lo\setminus\{\tilde{\boldsymbol{0}}\})$ to a suitable Poisson point
process.

We first prove a technical lemma which is also of independent interest, see \Cref{rem:MW16}.  Denote
by $\mathbb{S}=\{\tilde{\bx}\in\lo:\|\tilde{\bx}\|_\infty=1\}$ the unit sphere in $\lo$ and define
the polar decomposition
$\psi:\lo\setminus\{\tilde{\boldsymbol{0}}\}\mapsto (0,+\infty)\times \mathbb{S}$ with
$\psi(\tilde{\bx})=(\|\tilde{\bx}\|_\infty,\tilde{\bx}/\|\tilde{\bx}\|_\infty)$. 
\begin{lemma}
  \label{lem:RVofCluster}
  Under \Cref{hypo:AC}, the sequence $\nu_n=k_n \PP(\bsX_{n,1} \in \cdot)$ in
  $\mathcal{M}(\lo\setminus\{\tilde{\boldsymbol{0}}\})$ converges to
  $\nu=\theta\left(  d(- y^{-\alpha})\times \PP_{\bsQ}\right)\circ \psi$ in $w^\#$-topology and
  $\PP_{\bsQ}$ is the distribution of $\{Q_j\}$ defined in \eqref{eq:q}.
\end{lemma}

\begin{proof}
  Let $f$ be a bounded and continuous function on $\lo\setminus\{\tilde{\boldsymbol{0}}\}$ and
  $\epsilon>0$ such that $f(\tbx)=0$ if $\|\tbx\|_\infty \leq \epsilon.$ Then
  $\EE[f(\bsX_{n,1})] = \EE[f(\bsX_{n,1})\mathbbm{1}_{\{M_{1,r_n}>\epsilon a_n\}}],$ so by
  (\ref{eq:a}), (\ref{eq:def-candidate-extremalindex}) and \Cref{lem:conv-cluster} we get
  \begin{align*}
    \lim_{n\to\infty} \nu_n(f)&=\lim_{n\to\infty} k_n\EE[f(\bsX_{n,1})] \\
    & = \lim_{n\to\infty} n \PP(|X_0| >\epsilon a_n) \frac{\PP(M_{1,r_n}>\epsilon a_n)}
    {r_n\PP(|X_0| >\epsilon a_n)} \EE[f(\bsX_{n,1}) \mid M_{1,r_n}>\epsilon a_n]    \\
    & = \epsilon^{-\alpha}\theta\EE[f(\epsilon \bsZ)] \; .
  \end{align*}
  Applying \Cref{lem:conv-cluster}, the last expression is equal to
  $$
  \epsilon^{-\alpha}\theta\int_1^\infty\EE[f(\epsilon y\bsQ)]\alpha
  y^{-\alpha-1}dy=\theta\int_\epsilon^\infty\EE[f(y\bsQ)]\alpha y^{-\alpha-1} dy \; .
  $$
  Finally, since $\|\bsQ\|_\infty=1$ a.s. and $f(\tbx)=0$ if $\|\tbx\|_\infty \leq \epsilon$ we have
  that
  $$
  \lim_{n\to\infty} \nu_n(f)=\theta\int_0^\infty\EE[f(y\bsQ)]\alpha y^{-\alpha-1}dy=\nu(f) \; ,
  $$ 
  by definition of $\nu$.
\end{proof}

\begin{remark} \label{rem:MW16}
  The previous lemma is closely related to the large deviations result obtained in Mikosch and
  Wintenberger~\cite[Theorem 3.1]{mikosch:wintenberger:2016}. For a class of functions $f$ called
  cluster functionals, which can be directly linked to the functions we used in the proof of
  \Cref{lem:RVofCluster}, they showed that
  \begin{equation}
    \label{e:MW1}
    \lim_{n\to\infty}\frac{\EE[f(a_n^{-1}X_1,\dots,a_n^{-1}X_{r_n})]}{r_n\PP(|X_0| >a_n)}=\int_0^\infty
    \EE[f(y\{\Theta_t, t\geq 0\})-f(y\{\Theta_t, t\geq 1\})] \alpha y^{-\alpha-1} dy \;.
  \end{equation}
  However, for an arbitrary bounded measurable function $f:\lo \rightarrow \RR$ which is a.e. continuous with respect to $\nu$ and such that for some $\epsilon>0$,
$\|\tilde{\bx}\|_\infty\leq \epsilon$ implies $f(\tilde{\bx})=0$, \Cref{lem:RVofCluster} together with a continuous mapping argument and the fact that $k_n^{-1} \sim r_n\PP(|X_0| >a_n)$ yields
  \[
  \lim_{n\to\infty}\frac{\EE[f(a_n^{-1}X_1,\dots,a_n^{-1}X_{r_n})]}{r_n\PP(|X_0|
    >a_n)}=\lim_{n\to\infty} \nu_n(f)=\theta\int_0^\infty\EE[f(y\bsQ)]\alpha y^{-\alpha-1}dy\,,
  \]
  which
  gives an alternative and arguably more interpretable expression for the limit in \eqref{e:MW1}.
\end{remark}

To show convergence of $\PPC_n$ in
$\mathcal{M}_p([0,1]\times \lo\setminus\{\tilde{\boldsymbol{0}}\}),$ we will need to assume that, intuitively speaking, one can break
the dependent series $\{X_n, n\geq 1\}$ into asymptotically independent blocks. 

Recall the notion of truncation of an element $\tbx\in\lo$ at level $\epsilon>0$ denoted by
$\tbx_\epsilon$, see paragraph following \eqref{eq:portmanteau}.  Using similar arguments as in the
proof of \Cref{lem:conv-cluster}, it can be shown that the class of nonnegative functions $f$ on
$\lo\setminus\{\tilde{\boldsymbol{0}}\}$ which depend only on coordinates greater than some
$\epsilon>0$, i.e. they satisfy $ f(t,\tbx) = f(t, \tbx_\epsilon)$, and are continuous except maybe
on the set
$\{\tbx \in \lo\setminus\{\tilde{\boldsymbol{0}}\} : \|x_j\|=\epsilon \; \text{for some}\; j\in \ZZ
\;\text{where} \;(x_i)_{i\in\ZZ}\in \tbx \}$,
is convergence determining (in the sense of \Cref{rem:LaplaceSmallerFamily}).  We denote this class
by $\mathcal{F}_+$.

\begin{hypothesis}
  \label{hypo:Aprimecluster}
  There exists a sequence of integers $\{r_n,n\in\NN\}$ such that
  $\lim_{n\to\infty}r_n=\lim_{n\to\infty}n/r_n =\infty$ and
  \begin{align*}
    \lim_{n\to\infty} \left( \EE[\rme^{-\PPC_n(f)}] - \prod_{i=1}^{k_{n}} \EE \biggl[ \exp \biggl\{
      - f \biggl(\frac{i}{k_n},\bsX_{n,i} \biggr) \biggr\} \biggr] \right) = 0 \; ,
  \end{align*}
  for all $f\in \mathcal{F}_+$.
\end{hypothesis}

This assumption is somewhat stronger than the related conditions in \cite{davis:hsing:1995} or
\cite{basrak:krizmanic:segers:2012}, cf. Condition 2.2 in the latter paper, since we consider
functions of the whole cluster. Still, as we show in \Cref{lem:beta-implies-A}, $\beta$-mixing
implies \Cref{hypo:Aprimecluster}. Since sufficient conditions for $\beta$-mixing are well studied
and hold for many standard time series models (linear processes are notable exception, note), one
can avoid cumbersome task of checking the assumption above directly. Linear processes are considered
separately in \Cref{subsec:linear}.

It turns out that the choice of $\lo$ as the state space for clusters, together with the results
above, does not only preserve the order of the observations within the cluster, but also makes the
statement and the proof of the following theorem remarkably tidy and short.
\begin{theorem}
  \label{thm:PPconvInLo}
  Let $\sequence{X}$ be a stationary $\Rset^d$-valued regularly varying sequence with tail index $\alpha$,
  satisfying \Cref{hypo:AC,hypo:Aprimecluster} for the same sequence $\{r_n\}$. Then $\PPC_n \dto
  \PPC$ in $\mathcal{M}_p([0,1]\times \lo\setminus\{\tilde{\boldsymbol{0}}\})$ where $N''$ is a Poisson point process with intensity
  measure $Leb \times \nu$ which can be expressed as
  \begin{align}
    \PPC = \sum_{i=1}^\infty\delta_{(T_i, P_i\bsQ_{i})} \; , \label{eq:Repr}
  \end{align}
  where
  \begin{enumerate}[(i)]
  \item $\sum_{i=1}^\infty\delta_{(T_i,P_i)}$ is a Poisson point process on $[0,1]\times(0,\infty]$
    with intensity measure $Leb \times d(-\theta y^{-\alpha})$;
  \item $\{\bsQ_i,\, i\geq 1\}$ is a sequence of \iid\ elements in $\mathbb{S}$, independent of
    $\sum_{i=1}^\infty\delta_{(T_i, P_i)}$ and with common distribution equal to the distribution of
    $\bsQ$ in  \eqref{eq:q}.
  \end{enumerate}
\end{theorem}

\begin{proof}
  Let, for every $n\in\NN$, $\{\bsX_{n,i}^*,i=1,\dots,k_n\}$ be independent copies of~$\bsX_{n,1}$
  and define
  \begin{equation}
    N_n^*  = \sum_{i=1}^{k_n} \delta_{(i/k_n, \bsX_{n,i}^*)}\,.
  \end{equation} 
  Since by the previous lemma, $k_n \PP(\bsX_{n,1} \in \cdot)\rightarrow_{w^\#}\nu$ in
  $\mathcal{M}(\lo\setminus\{\tilde{\boldsymbol{0}}\})$, the convergence of $N_n^*$ to $N''$ in
  $\mathcal{M}_p([0,1]\times\lo\setminus\{\tilde{\boldsymbol{0}}\})$ follows by an straightforward
  adaptation of \cite[Proposition~3.21]{resnick:1987} (cf. \cite[Lemma 2.2.(1)]{davis2008extreme},
  see also \cite[Theorem~2.4]{dehaan:lin:2001} or
  \cite[Proposition~2.13]{roueff:soulier:2015}). \Cref{hypo:Aprimecluster} now yields that $\PPC_n$
  converge in distribution to the same limit since the convergence determining family
  $\mathcal{F}_+$ consists of functions which are a.e. continuous with respect to the measure
  $Leb\times \nu$. Finally, the representation of $\PPC$ follows easily by standard Poisson point
  process transformation argument (see~\cite[Section 3.3.2.]{resnick:1987}).
\end{proof}

Under the conditions of \Cref{thm:PPconvInLo}, as already noticed in
\cite[Remark~4.7]{basrak:segers:2009}, $\theta=\PP(\sup_{i\geq1} \|Y_i\|\leq1)$ is also the extremal
index of the time series $\{\|X_j\|\}$, \ie\
$\lim_{n\to\infty} \PP(a_n^{-1}\max_{1\leq i \leq n} \|X_i\| \leq x) = \rme^{-\theta x^{-\alpha}}$ for
all $x>0$.
  
\begin{remark}
  Note that the restriction to the time interval $[0,1]$ is arbitrary and it could be substituted by
  an arbitrary closed interval $[0,T]$.  
\end{remark}

\subsection{Linear processes}
\label{subsec:linear}
As we observed above, $\beta$-mixing offers a way of establishing \Cref{hypo:Aprimecluster} for a
wide class of time series models.  However, for linear processes, the truncation method offers an
alternative and simpler way to obtain the point process convergence stated in the previous theorem.

Let $\{\xi_t,t\in\ZZ\}$ be a sequence of \iid\ random variables with regularly varying distribution
with index $\alpha >0$. Consider the linear process $\{X_t,t\in\ZZ\}$ defined by
\begin{align*}
  X_t = \sum_{j\in\ZZ} c_j \xi_{t-j} \; , 
\end{align*}
where $\{c_j,j\in\ZZ\}$ is a sequence of real numbers such that $|c_0|>0$ and
\begin{align}   
  \label{eq:condition-linear}
  \sum_{j\in\ZZ} |c_j|^\delta < \infty \mbox{ with }
  \begin{cases}
    \delta<\alpha & \mbox{ if } \alpha\in      (0,1]  \; , \\
    \delta<\alpha & \mbox{ if } \alpha\in      (1,2] \mbox{ and }  \EE[\xi_0]=0 \; , \\
    \delta=2 & \mbox{ if } \alpha>2 \mbox{ and }  \EE[\xi_0]=0 \; , \\
\delta = 1 & \mbox{ if } \alpha>1 \mbox{ and }  \EE[\xi_0]\ne0 \; .
  \end{cases}
\end{align}
These conditions imply that $\sum_{j\in\ZZ} |c_j|^\alpha<\infty$. Furthermore, it has been 
proved in
\cite[Lemma~A3]{mikosch:samorodnitsky:2000} that the sequence $\{X_t,t\in\ZZ\}$ is well defined,
stationary and regularly varying with tail index $\alpha$ and
\begin{align}
  \lim_{u\to\infty} \frac{\PP(|X_0|>u)}{\PP(|\xi_0|>u)}=\sum_{j\in\ZZ} |c_j|^\alpha \; .  \label{eq:tail-equivalence-linear}
\end{align}
\cite[Lemma~A.4]{mikosch:samorodnitsky:2000} proved that for $\alpha\in(0,2]$, it is possible to
take $\delta=\alpha$ in (\ref{eq:condition-linear}) at the cost of some restrictions on the
distribution of $\xi_0$ which are satisfied for Pareto and stable distributions. 

The spectral tail process $\{\Theta_t\}$ of the linear process was computed in \cite{meinguet:segers:2010}.  It can be
described as follows: let $\Theta^\xi$ be an $\{-1,1\}$-valued random variable with distribution
equal to the spectral measure of $\xi_0$. Then
\begin{align}
  \mathcal{L}\left(\{\Theta_t, t\in\ZZ \}\right) 
  = \mathcal{L}\left(\left\lbrace\frac{c_{t+\intrv}}{|c_\intrv|} \Theta^\xi,t\in\ZZ\right\rbrace\right), \label{eq:spectralPr_of_linearPr}
\end{align} 
where $\intrv$ is an integer valued random variable, independent of $\Theta^\xi$, such that
\begin{align}
  \PP(\intrv=n) = \frac{|c_n|^\alpha} {\sum_{j\in\ZZ} |c_j|^\alpha},\; n\in\ZZ \; . \label{eq:def-intrv}
\end{align}
It is also proved in \cite{meinguet:segers:2010} that  the coefficient 
$\theta$ from~(\ref{eq:def-theta}) is given by
\begin{align*}
  \theta = \frac{\max_{j\in\ZZ} |c_j|^\alpha}{\sum_{j\in\ZZ} |c_j|^\alpha} \; .
\end{align*}
Since $\lim_{|j|\to\infty} c_j=0$, 
random element  $\bsQ = \sequence{Q}[\ZZ][j]$ in
the space $\mathbb{S}$ introduced in (\ref{eq:q}) is well defined and given by
\begin{equation}\label{eq:Qlin}
\bsQ = \left\{ \frac{\Theta^\xi c_j}{\max_{i\in\ZZ}|c_i|} \; , j\in\ZZ
\right\} \; .
\end{equation}

The following proposition can be viewed as an extension of \cite[Theorem 2.4]{davis:resnick:1985l} and also as a version of \Cref{thm:PPconvInLo} adapted to linear processes.

\begin{proposition}
  \label{prop:ppconv-linear}
  Let $\{r_n\}$ be a nonnegative non decreasing integer valued sequence such that
  $\lim_{n\to\infty} r_n =\lim_{n\to\infty}n/r_n =\infty$ and let $\{a_n\}$ be a non decreasing sequence such
  that $n\PP(|X_0|>a_n)\to 1$. Then
  \begin{equation}
    \label{eq:MAinfPPC-linear}
    \PPC_n=    \sum_{i=1}^{k_n} \delta_{(i/k_n,(X_{(i-1)r_n+1},\dots,X_{ir_n})/a_n})\dto \sum_{i=1}^\infty\delta_{(T_i, P_i \bsQ_{i})} \;
  \end{equation}
 in $\mathcal{M}_p([0,1] \times \lo\setminus\{\tilde{\boldsymbol{0}}\})$ where $\sum_{i=1}^\infty\delta_{(T_i,P_i)}$ is a Poisson point process on $[0,1]\times(0,\infty]$
  with intensity measure $Leb \times d(-\theta y^{- \alpha})$, independent of the \iid\ sequence
  $\{\bsQ_{i},\,i\geq 1\}$ with values in $\mathbb{S}$ and common distribution equal to the distribution of
  $\bsQ$ in \eqref{eq:Qlin}.

\end{proposition}

\begin{proof}
  The proof of~(\ref{eq:MAinfPPC}) is based on a truncation argument which compares $X_t$ with
 \[
 X_t^{(m)} = \sum_{j=-m}^{m} c_j \xi_{t-j} \; , \ \ t\in\ZZ \; .
 \]
 Let $\{b_n\}$ and $\{a_{m,n}\}$ be non decreasing sequences such that $n\PP(|Z_0|>b_n)\to 1$ and
 $n\PP(|X_0^{(m)}|>a_{m,n})\to 1$. The limit~(\ref{eq:tail-equivalence-linear}) implies that
 $a_{m,n}\sim \left(\sum_{j=-m}^m |c_j|^\alpha\right)^{1/\alpha}b_n$.  Let~$\PPC_{m,n}$ be the point
 process of exceedences of the truncated sequence defined by
 \begin{align*}
   \PPC_{m,n} = \sum_{i=1}^{k_n} \delta_{(i/k_n,(X_{(i-1)r_n+1}^{(m)},\dots,X_{ir_n}^{(m)})/b_n)} \; .
 \end{align*}
 The process $\{X_t^{(m)}\}$ is $(2m)$-dependent (hence $\beta$-mixing with $\beta_j=0$ for $j>2m$) and
 therefore satisfies the conditions of \Cref{thm:PPconvInLo}.  Thus $\PPC_{m,n}\dto\PPC_{(m)}$ with
 \begin{align*}
   \PPC_{(m)} = \sum_{i=1}^\infty  \delta_{(T_i,P_i\Theta_i^\xi\{c_j^{(m)}\})} \; , 
 \end{align*}
 with $\sum_{i=1}^\infty \delta_{(T_i,P_i)}$ a Poisson point process on $[0,1]\times(0,\infty)$ with
 mean measure $Leb\times \rmd(-y^{-\alpha})$, $c_j^{(m)} = c_j$ if $|j|\leq m$ and $c_j=0$
 otherwise, and $\Theta_i^\xi$, $i\geq1$ are \iid\ copies of $\Theta^\xi$. Since $\{c_j^{(m)}\}$
 converges to $\{c_j\}$ in $\lo$, it follows that 
 \begin{equation*}
   \PPC_{(m)} \to    \PPC_\infty = \sum_{i=1}^\infty\delta_{(T_i, P_i \Theta_i^\xi \{c_j\})}  \;  , 
 \end{equation*}
almost surely in $\mcm_p([0,1]\times\lo\setminus\{\boldsymbol{0}\})$. 

Define now 
\begin{align*}
  \PPC_{\infty,n} = \sum_{i=1}^{k_n} \delta_{(i/k_n,(X_{(i-1)r_n+1},\dots,X_{ir_n})/b_n)} \; .
\end{align*}
Then, for every bounded Lipschitz continuous function $f$ defined on
$[0,1]\times \lo\setminus\{\boldsymbol{0}\}$ with bounded support,
\begin{align*}
  \lim_{m\to\infty}  \limsup_{n\to\infty} \PP(|\PPC_{n,m}(f)-\PPC_{\infty,n}(f)|>\eta) = 0 \; .
\end{align*}
As in the proof of \cite[Theorem 2.4]{davis:resnick:1985l}, this follows from 
\begin{align}
  \label{eq:lemmaDR1985}
   \lim_{m\to\infty}\limsup_{n\to\infty} \PP(\max_{1\leq i \leq n}|X_i^{(m)}-X_i|>b_n\gamma)=0 \; . 
\end{align}
for all $\gamma>0$ which is implied by (\ref{eq:tail-equivalence-linear}); see
\cite[Lemma~2.3]{davis:resnick:1985l}. 

This proves that
\begin{align}
  \PPC_{\infty,n}\dto\PPC_\infty  \; .    \label{eq:MAinfPPC}
\end{align}
and since $\PPC_n$ and $\PPC_{\infty,n}$ differ
only by a deterministic scaling of the points, this proves our result.
\end{proof}

\section{Convergence of the partial sum process}
\label{sec:consum}
In order to study the convergence of the partial sum process in cases where it fails to hold in the
usual space $D$, we first introduce an enlarged space $E$. In the rest of the paper we restrict to
the case of $\RR$-valued time series.
\subsection{The space of decorated \cadlag\ functions}
To establish convergence of the partial sum process of a dependent sequence $\{X_n\}$ we will
consider the function space $E\equiv E([0,1],\RR)$ introduced in Whitt~\cite[Sections 15.4 and
15.5]{whitt:2002}. For the benefit of the reader, in what follows, we briefly introduce this space
closely following the exposition in the previously mentioned reference.

The elements of $E$ have the form
\[
(x , \setJ , \{I(t):t\in \setJ\})
\] 
where 
\begin{enumerate}[-]
\item $x\in  D([0,1],\RR)$;
\item $\setJ$ is a countable subset of $[0,1]$ with $Disc(x)\subseteq \setJ$, where $Disc(x)$ is the
  set of discontinuities of the c\`adl\`ag function $x$;
\item for each $t\in \setJ$, $I(t)$ is a closed bounded interval (called the decoration) in $\RR$
  such that $x(t),x(t-)\in I(t)$ for all $t\in \setJ$.
\end{enumerate}
Moreover, we assume that for each $\epsilon>0,$ there are at most finitely many times $t$ for which
the length of the interval $I(t)$ is greater than $\epsilon.$ This ensures that the graphs of
elements in $E,$ defined below, are compact subsets of $\RR^2$ which allows one to impose a metric
on $E$ by using the Hausdorff metric on the space of graphs of elements in $E$.

Note that every triple $(x,\setJ,\{I(t):t\in \setJ\})$ can be equivalently represented by a set-valued function
\[
x'(t):=
\begin{cases}
I(t) & \mbox{ if } t \in \setJ \; , \\
\{x(t)\} &  \mbox{ if } t\not\in \setJ \; ,
\end{cases}
\]
or by the graph of $x'$ defined by
\[
\Gamma_{x'}:=\{(t,z)\in [0,1]\times\RR:z\in x'(t)\}.
\]
In the sequel, we will usually denote the elements of $E$ by $x'$. 

Let $\hausdorff$ denote the Hausdorff metric on the space of compact subsets of $\RR^d$ (regardless
of dimension) i.e.~for compact subsets $A,B,$
\begin{equation*}
\hausdorff(A,B)=\sup_{x\in A}\|x-B\|_{\infty} \vee \sup_{y\in B}\|y-A\|_{\infty} \; , 
\end{equation*}
where $\|x-B\|_{\infty}=\inf_{y\in B}\|x-y\|_\infty$.  We then define a metric on $E$, denoted by
$\metricE$, by
\begin{equation}\label{eq:HausE}
\metricE(x',y') = \hausdorff(\Gamma_{x'},\Gamma_{y'}) \; .
\end{equation}
We call the topology induced by $\metricE$ on $E$ the $M_2$ topology. This topology is separable,
but the metric space $(E,\metricE)$ is not complete. Also, we define the {uniform metric} on $E$ by
\begin{equation}\label{e:unmetric1}
  \unifE(x',y')=\sup_{0\leq t \leq 1} \hausdorff(x'(t),y'(t)) \; ,
\end{equation}
Obviously, $\unifE$ is a stronger metric than $\metricE$, i.e. for any $x',y'\in E$,
\begin{equation}
  \label{e:unmetric2}
  \metricE(x',y')\leq \unifE(x',y').
\end{equation}
We will often use the following elementary fact: for $a\leq b$ and $c\leq d$ it holds that
\begin{equation}
  \label{eq:HasusdorffforInterval} 
  \hausdorff([a,b],[c,d])\leq |c-a|\vee |d-b|.
\end{equation}
By a slight abuse of notation, we identify every $x\in D$ with an element in $ E$ represented
by
\[
(x  ,  Disc(x)  , \{[x(t-),x(t)]: t\in Disc(x) \}) \; ,
\]
where for any two real numbers $a,b$ by $[a,b]$ we denote the closed interval
$[\min\{a,b\},\max\{a,b\}]$. Consequently, we identify the space $D$ with the subset $D'$ of $E$
given by
\begin{equation}
  \label{D'}
  D'=\{x'\in E:\setJ=Disc(x)\; \text{and for all} \; t\in \setJ, \; I(t)=[x(t-),x(t)]\} \; .
\end{equation}
For an element $x'\in D'$ we have
\[
\Gamma_{x'}=\Gamma_{x},
\]
where $\Gamma_{x}$ is the completed graph of $x$.  Since the $M_2$ topology on $D$ corresponds to
the Hausdorff metric on the space of the completed graphs $\Gamma_{x}$, the map
$x\to (x , Disc(x) , \{[x(t-),x(t)]: t\in Disc(x) \})$ is a homeomorphism from $D$ endowed with the
$M_2$ topology onto $D'$ endowed with the $M_2$ topology. This yields the following lemma. 
\begin{lemma}
  \label{lem:homeomorphismD} 
  The space $D$ endowed with the $M_2$ topology is homeomorphic to the subset $D'$ in $E$ with the
  $M_2$ topology.
\end{lemma} 

\begin{remark}
  \label{rem:additionInE}
  Because two elements in $E$ can have intervals at the same time point, addition in $E$ is in
  general not well behaved. However, problems disappear if one of the summands is a continuous
  function. In such a case, the sum is naturally defined as follows: consider an element
  $x'=(x,\setJ,\{I(t) : t\in \setJ\})$ in $E$ and a continuous function $b$ on $[0,1],$ we define
  the element $x'+b$ in $E$ by
  $$
  x'+b = (x+b,\setJ,\{I(t)+b(t): t\in \setJ\}) \; .
  $$
\end{remark}

We now state a useful characterization of convergence in $(E,\metricE)$ in terms of the local-maximum
function defined for any $x'\in E$ by
\begin{equation} 
  \label{e:locmax}
  M_{t_1,t_2}(x'):=\sup\{z:z\in x'(t), t_1\leq t \leq t_2\},
\end{equation} 
for $0\leq t_1 < t_2 \leq 1$. 

\begin{theorem}[Theorem 15.5.1 Whitt \cite{whitt:2002}]
  \label{thm:charofconvE}
  For elements $x_n',x'\in E$ the following are equivalent:
  \begin{itemize}
  \item[(i)] $x_n'\rightarrow x'$ in $(E,\metricE),$ i.e. $\metricE(x_n',x')\rightarrow 0.$
  \item[(ii)] For all $t_1<t_2$ in a countable dense subset of $[0,1],$ including $0$ and $1,$
    \[
    M_{t_1,t_2}(x_n')\rightarrow M_{t_1,t_2}(x') \text{ in } \RR
    \]
    and
    \[
    M_{t_1,t_2}(-x_n')\rightarrow M_{t_1,t_2}(-x') \text{ in } \RR.
    \]
  \end{itemize}
\end{theorem}

\subsection{Invariance principle in the space $E$}
Consider the partial sum process in $D([0,1])$ defined by
\[
S_n(t) = \sum_{i=1}^{\lfloor nt \rfloor} \frac{X_i}{a_n} \; , \; t\in[0,1],
\]
and define also
\begin{align*}
  V_n(t)=
  \begin{cases}
    S_n(t)  & \mbox{ if } 0< \alpha < 1 \; ,  \\
    S_n(t)-\lfloor nt \rfloor \EE\left(\frac{X_1}{a_n}\1{\{|X_1|/a_n\leq 1\}}\right) & \mbox{ if } 1
    \leq \alpha < 2 \; . \end{cases}
\end{align*}
As usual, when $1 \leq \alpha <2$, an
additional condition is needed to deal with the small jumps. 
\begin{hypothesis}
  \label{hypo:ANSJ} 
  For all $\delta>0$,
  \begin{align}
    \label{eq:ANSJ}
    \lim_{\epsilon\to0} \limsup_{n\to\infty} \PP \biggl( \max_{1 \leq k \leq n} \biggl| \sum_{i=1}^k
    \{X_{i}\1{\{|X_i|\leq a_n \epsilon\}} -\EE[X_i \1{||X_{i}| \leq a_n\epsilon\}}] \} \biggr| >
    a_n\delta \biggr) = 0 \; .
  \end{align}
\end{hypothesis}
It is known from \cite{davis:hsing:1995} that the finite dimensional marginal distributions of $V_n$
converge to those of an $\alpha-$stable L\'evy process. This result is strengthened in
\cite{basrak:krizmanic:segers:2012} to convergence in the $M_1$ topology if $Q_{j}Q_{j'}\geq0$ for
all $j\ne j'\in\ZZ,$ i.e. if all extremes within one cluster have the same sign. In the next
theorem, we remove the latter restriction and establish the convergence of the process $V_n$ in the
space $E$.  

For that purpose,  we assume only regular variation of the sequence $\sequence{X}$ and  the  conclusion of \Cref{thm:PPconvInLo}, i.e.
\begin{align}
\PPC_n\dto \PPC  = \sum_{i=1}^\infty\delta_{(T_i, P_i\bsQ_{i})}  \label{eq:Repr-sum}
\end{align}
in $\mathcal{M}_p([0,1]\times \lo\setminus\{\tilde{\boldsymbol{0}}\})$, where
$\sum_{i=1}^\infty\delta_{(T_i,P_i)}$ is a Poisson point process on $[0,1]\times(0,\infty]$ with
intensity measure $Leb \times d(-\theta y^{-\alpha})$ with $\theta>0$ and
$\bsQ_i=\{Q_{i,j},j\in\ZZ\}$, $i\geq1$ are \iid\ sequences in $\lo$, independent of
$\sum_{i=1}^\infty\delta_{(T_i, P_i)}$ and such that $\PP(\sup_{j\in\ZZ} |Q_{1,j}|=1)=1$.  Denote by
$\bsQ=\{Q_j,j\in\ZZ\}$ a random sequence the with the distribution equal to the distribution of
$\bsQ_1$.  We also describe the limit of $V_n$ in terms of the point process $\PPC$.

The convergence \eqref{eq:Repr-sum} and Fatou's lemma imply that
$\theta \sum_{j\in\ZZ} \EE[|Q_j|^\alpha] \leq 1$, as originally noted by
\cite[Theorem~2.6]{davis:hsing:1995}.  This implies that for $\alpha\in(0,1]$,
\begin{align}
  \label{eq:SumOfTheQjs}
  \EE \left[ \left( \sum_{j\in\ZZ} |Q_j| \right)^\alpha \right] < \infty \; .
\end{align}
For $\alpha>1$, this will have to be assumed. Furthermore, the case $\alpha=1$, as usual,
  requires additional care. We will assume that
\begin{align}
  \label{eq:SumOfTheQjs_alpha_equal_to_1}
  \EE \left[ \sum_{j\in\ZZ} |Q_j| \log\left(|Q_j|^{-1} \sum_{i\in\ZZ} |Q_i|\right) \right] < \infty \; ,
\end{align}
where we use the convention $|Q_j| \log\left(|Q_j|^{-1} \sum_{i\in\ZZ} |Q_i|\right)=0$ if $|Q_j|=0$.
Fortunately, it turns out that conditions (\ref{eq:SumOfTheQjs}) and
(\ref{eq:SumOfTheQjs_alpha_equal_to_1}) are satisfied in most examples. See
\Cref{rem:conditions_in_terms_of_spectral_process} below.

\begin{theorem} 
  \label{thm:PartialSumConvinE}
  Let $\sequence{X}$ be a stationary $\RR$-valued regularly varying sequence with tail index
  $\alpha\in(0,2)$ and assume that the convergence in \eqref{eq:Repr-sum} holds.  If $\alpha\geq1$
  let \Cref{hypo:ANSJ} hold. For $\alpha>1$, assume that (\ref{eq:SumOfTheQjs}) holds, and for
  $\alpha=1$, assume that (\ref{eq:SumOfTheQjs_alpha_equal_to_1}) holds.  Then
  \[
  V_n\dto V'=\left(V , \{T_i \}_{i\in\NN} , \{I(T_i)\}_{i\in \NN} \right) \; ,
  \] 
  with respect to $M_2$ topology on $E([0,1],\RR)$, where 
  \begin{itemize}
  \item[(i)] $V$ is an $\alpha-$stable L\'evy process on $[0,1]$ given by
    \begin{subequations}
      \label{eq:levylimit}
      \begin{align}
        V(\cdot) & = \sum_{T_i\leq \cdot}\sum_{j\in\ZZ} P_iQ_{i,j} \; ,  \ 0 < \alpha < 1 \; ,  \label{eq:levylimit-alpha<1} \\
        V(\cdot) & = \lim_{\epsilon\to 0} \left( \sum_{T_i\leq \cdot}\sum_{j\in\ZZ}
                   P_iQ_{i,j}\1{\{|P_iQ_{i,j}|>\epsilon\}} - (\cdot) \int_{\epsilon<|x|\leq 1} x \mu(\rmd x) \right)
                   \;, \ 1 \leq \alpha < 2 \; ,  \label{eq:levylimit-alphageq1}
      \end{align}
    \end{subequations}
    where the series in (\ref{eq:levylimit-alpha<1}) is almost surely absolutely summable and the
    holds uniformly on $[0,1]$ almost surely {(along some subsequence)} with $\mu$ given in
    \eqref{eq:def-mu}.
  \item[(ii)] For all $i\in\NN,$
    \[
    I(T_i)=V(T_i-)+P_i\left[\inf_{k\in\ZZ}\sum_{j\leq k} Q_{i,j} \; ,\sup_{k\in\ZZ}\sum_{j\leq k}
      Q_{i,j}\right].
    \]
  \end{itemize}
\end{theorem}

Before proving the theorem, we make several remarks. We first note that for $\alpha<1$, convergence
of the point process is the only assumption of theorem. Further, an extension of
\Cref{thm:PartialSumConvinE} to multivariate regularly varying sequences would be possible at the
cost of various technical issues (similar to those in \cite[Section~12.3]{whitt:2002} for the
extension of $M_1$ and $M_2$ topologies to vector valued processes) and one would moreover need to
alter the definition of the space $E$ substantially and introduce a new and  weaker notion of $M_2$
topology.

\begin{remark}
  \label{rem:useful-W}
  If (\ref{eq:SumOfTheQjs}) holds, the sums $W_i=\sum_{j\in \ZZ}|Q_{i,j}|$ are almost surely
  well-defined and $\{W_i, i\geq1\}$ is a sequence of \iid\ random variables with $\EE[W_i^\alpha]
  <\infty$. Furthermore, by independence of $\sum_{i=1}^\infty \delta_{P_i}$ and $\{W_i\},$ it
  follows that $\sum_{i=1}^\infty\delta_{P_i W_i}$ is a Poisson point process on $(0,\infty]$ with
  intensity measure $\theta \EE [W_1^\alpha] \alpha y^{-\alpha -1}dy.$ In particular, for every
  $\delta>0$ there a.s. exist at most finitely many indices $i$ such that $P_iW_i >\delta$. Also,
  this implies that
  \begin{equation}
    \label{eq:negligibilityCond1}
    \sup_{i\in \NN} \sum_{j\in \ZZ} P_i|Q_{i,j}|\1{\{P_i|Q_{i,j}|\leq \epsilon\}} \rightarrow 0
  \end{equation}
  almost surely as $\epsilon\rightarrow 0.$ These facts will be used several times in the proof.
\end{remark}

\begin{remark}\label{rem:charac-V}
The L\'evy process $V$ from \Cref{thm:PartialSumConvinE} is the weak limit in the sense of finite
  dimensional distributions of the partial sum process $V_n$, characterized by
  \begin{align}
    \label{eq:charac-V}
    \log \EE[\rme^{\rmi z V(1)}] = 
    \begin{cases}
      \rmi \location z + \Gamma(1-\alpha)\cos(\pi\alpha/2)
      \sigma^\alpha|z|^\alpha\{1 - \rmi \skewness \mathrm{sgn}(z) \tan(\pi\alpha/2)\}   & \alpha\ne1 \; , \\
      \rmi \location z -
      \frac\pi2 \sigma|z|\{1 - \rmi \frac2\pi\skewness \mathrm{sgn}(z) \log(|z|)\}   & \alpha=1 \; ,
    \end{cases}
  \end{align}
with, denoting $x^{\langle\alpha\rangle} = x|x|^{\alpha-1}=x_+^\alpha - x_{-}^\alpha$,
\begin{align*}
  \sigma^\alpha & = \theta\EE\left[\left|\sum_{j\in \ZZ}Q_j\right|^\alpha\right] \; , \ \
  \skewness  = \frac{\EE[(\sum_{j\in \ZZ}Q_j)^{\langle \alpha \rangle}]} {\EE[|\sum_{j\in \ZZ}Q_j|^\alpha]}
\end{align*}
and 
\begin{enumerate}[(i)]
\item $\location=0$ if $\alpha< 1$;
\item $\location=(\alpha-1)^{-1}\alpha \theta\EE\left[\sum_{j\in \ZZ} Q_j^{\langle \alpha \rangle}\right]$ if $\alpha>1$; 
\item if $\alpha=1$, then
\[
\location = \theta\left(c_0 \EE\left[\sum_{j\in\ZZ} Q_j\right] -  \EE\left[\sum_{j\in\ZZ} Q_j
  \log\left(\left| \sum_{j\in\ZZ} Q_j \right|\right)\right]-\EE\left[\sum_{j\in\ZZ} Q_j
  \log\left(\left|Q_{j} \right|^{-1}\right)\right]\right) \; ,
\]
with $c_0 = \int_{0}^{\infty}(\sin y - y\1{(0,1]}(y))y^{-2} \rmd y$.
\end{enumerate}  
These parameters were computed in \cite[Theorem 3.2]{davis:hsing:1995} but with a complicated
expression for the location parameter $\location$ in the case $\alpha=1$ (see \cite[Remark
3.3]{davis:hsing:1995}). The explicit expression given here, which holds under the
assumption~(\ref{eq:SumOfTheQjs_alpha_equal_to_1}), is new; the proof is given in
\Cref{lem:alpha_stable_parameters}. As often done in the literature, if the sequence is assumed to
be symmetric then assumption~(\ref{eq:SumOfTheQjs_alpha_equal_to_1}) is not needed and the location
parameter is 0.
\end{remark}

\begin{remark}
  \label{rem:conditions_in_terms_of_spectral_process}
  \cite{planinic:soulier:2017} showed that whenever $\PP(\lim_{|t|\to \infty}|Y_t|=0)=1$, the
  quantity $\theta$ from (\ref{eq:def-theta}) is positive and
  $\theta=\PP(\sup_{i\leq -1}|Y_i|\leq 1)$. In particular, the sequence $\bsQ$ from (\ref{eq:q}) is
  well defined in this case and moreover, by \cite[Lemma 3.11]{planinic:soulier:2017}, the condition
  \eqref{eq:SumOfTheQjs} turns out to be equivalent to
  \begin{align}\label{eq:SumOfTheThetajs}
    \EE \left[ \left( \sum_{j=0}^\infty |\Theta_j| \right)^{\alpha-1} \right] < \infty \; 
  \end{align}
  which is automatic if $\alpha\in (0,1]$. Furthermore, if $\alpha=1$,
  \cite[Lemma~3.14]{planinic:soulier:2017} shows that the condition
  \eqref{eq:SumOfTheQjs_alpha_equal_to_1} is then equivalent to
  \begin{align}
    \label{eq:LogSumOfTheThetajs}
    \EE \left[ \log \left( \sum_{j=0}^\infty |\Theta_j| \right) \right] < \infty \; .
  \end{align}  
  These conditions are easier to check than conditions \eqref{eq:SumOfTheQjs} and
  \eqref{eq:SumOfTheQjs_alpha_equal_to_1} since it is easier to determine the distribution of the
  spectral tail process than the distribution of the process $\bsQ$ from (\ref{eq:q}). In fact, it
  suffices to determine only the distribution of the forward spectral tail process
  $\{\Theta_j,j\geq 0\}$ which is often easier than determining the distribution of the whole
  spectral tail process. For example, it follows from the proof of
  \cite[Theorem~3.2]{mikosch:wintenberger:2014} that for functions of Markov chains satisfying a
  suitable drift condition, (\ref{eq:SumOfTheThetajs}) and (\ref{eq:LogSumOfTheThetajs}) hold {for
    all $\alpha>0$.} Also, notice that for the linear process $\{X_t\}$ from \Cref{subsec:linear}
  these conditions are satisfied if $\sum_{j\in \ZZ} |c_j|<\infty$.
  
  Moreover, \cite[Corollary~3.12 and Lemma~3.14]{planinic:soulier:2017} imply that the scale,
  skewness and location parameters from \Cref{rem:charac-V} can also be expressed in terms of the
  forward spectral tail process as follows:
  \begin{align*}
    \sigma^\alpha & = \EE \left[ \left|\sum_{j=0}^\infty \Theta_j\right|^\alpha - \left|\sum_{j=1}^\infty \Theta_j\right|^\alpha  \right] \; , \\
    \skewness  & = \sigma^{-\alpha} \EE \left[ \left(\sum_{j=0}^\infty \Theta_j\right)^{\langle\alpha\rangle} 
                 - \left(\sum_{j=1}^\infty \Theta_j\right)^{\langle\alpha\rangle}  \right] \; ,
  \end{align*}
   $\location=0$ if $\alpha< 1$, $\location = (\alpha-1)^{-1} \alpha \EE[\Theta_0]$ if $\alpha>1$ (see (\ref{eq:DH95identity}))
  and 
  \begin{align*}
    \location &  = c_0 \EE[\Theta_0]  - \EE\left[\sum_{j=0}^\infty\Theta_j \log\left(\left|\sum_{j=0}^\infty\Theta_j\right|\right)
                -\sum_{j=1}^\infty\Theta_j \log\left(\left|\sum_{j=1}^\infty\Theta_j\right|\right)\right] \; ,
  \end{align*}
  if $\alpha=1$.  It can be shown that these expressions coincide for $\alpha\ne1$ with those in the
  literature, see \eg\ \cite[Theorem~4.3]{mikosch:wintenberger:2016}. As already noted, the
  expression of the location parameter for $\alpha=1$ under the
  assumption~(\ref{eq:SumOfTheQjs_alpha_equal_to_1}) (or (\ref{eq:LogSumOfTheThetajs})) is new.
\end{remark}

\begin{example}
  \label{ex:MAinf4}
  Consider again the linear process $\{X_t\}$ from \Cref{subsec:linear}.  For infinite order moving
  average processes, \cite{davis:resnick:1985l} proved convergence of the finite dimensional
  distributions of the partial sum process; \cite{avram:taqqu:1992} proved the functional
  convergence in the $M_1$ topology (see \cite[Section~12.3]{whitt:2002}) when $c_j\geq0$ for all
  $j$; using the $S$ topology (which is weaker than the $M_1$ topology and makes the supremum
  functional not continuous), \cite{balan:jakubowski:louhichi:2016} proved the corresponding result
  under more general conditions on the sequence $\{c_j\}$ in the case $\alpha\leq 1$.

  Our \Cref{thm:PartialSumConvinE} directly applies to the case of a finite order moving average
  process.  To consider the case of an infinite order moving average process, we assume for
  simplicity that $\alpha<1$.  Applying \Cref{thm:PartialSumConvinE} to the point process
  convergence in \eqref{eq:MAinfPPC-linear}, one obtains the convergence of the partial sum process
  $V_n\dto V'=\left(V , \{T_i \}_{i\in\NN} , \{I(T_i)\}_{i\in \NN} \right) $ in $(E,[0,1])$ where
\[
V(\cdot)=\frac{\sum_{j\in\ZZ} c_j}{\max_{j\in\ZZ} |c_j| } \sum_{T_i\leq \cdot}P_i \Theta_i^\xi\,,
\]
and 
 \[
 I(T_i)= V(T_i-)+\frac{P_i\Theta_i^\xi}{\max_{j\in\ZZ} |c_j| }
 \left[\inf_{k\in\ZZ}
 \sum_{j\leq k } c_j \; ,\sup_{k\in\ZZ} \sum_{j\leq k } c_j\right].
 \]
  For an illustration consider the process
  \begin{align*}
  X_t =  \xi_{t} + c \xi_{t-1} \; .
  \end{align*}
   In the case $c \geq 0$, the convergence of partial sum process in $M_1$ topology follows from
  \cite{avram:taqqu:1992}. On the other hand, for negative $c$'s convergence
  fails in any of Skorohod's topology, but partial sums do have a limit in the sense described by
  our theorem as can be also guessed from Figure~\ref{FigMA1}. 
\end{example}

\begin{figure}[h]
  \centering
  \includegraphics[width=0.8\linewidth]{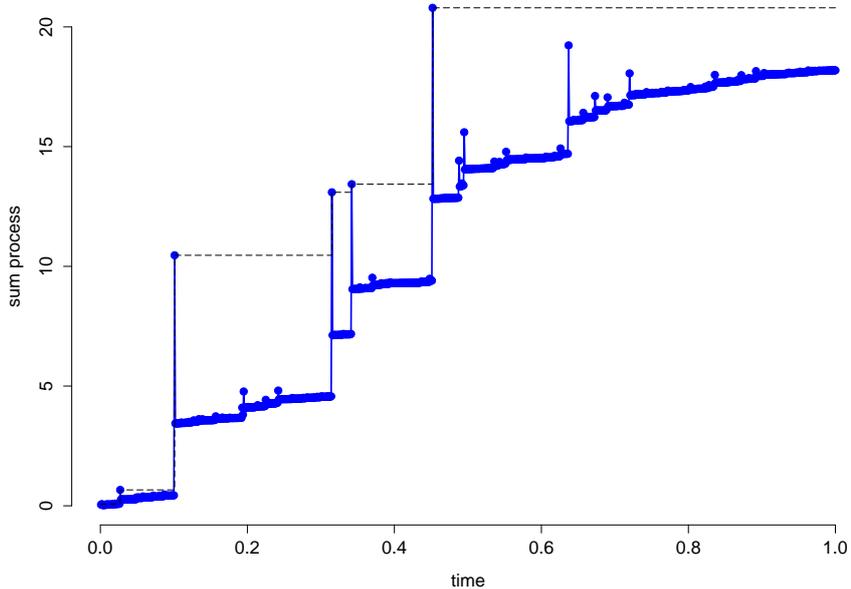}
  \caption{A simulated sample path of the process $S_n$ in the case of linear sequence
    $X_t = \xi_{t} - 0.7 \xi_{t-1}$ with index of regular variation $\alpha=0.7$ in blue.  Observe
    that due to downward ``corrections'' after each large jump, in the limit the paths of the
    process $S_n$ cannot converge to a \cadlag\ function.}
  \label{FigMA1}
\end{figure}

\begin{remark}
  We do not exclude the case $\sum_{j\in\ZZ}Q_{j}=0$ with probability one, as happens for instance
  in \Cref{ex:MAinf4} with $c=-1$. In such a case, the \cadlag\ component $V$ is simply
  the null process.
\end{remark}   
 
\begin{example}
  Consider a stationary GARCH$(1,1)$ process
  \[
  X_t=\sigma_t Z_t,\;\; \sigma_t^2=\alpha_0 + \alpha_1X_{t-1}^2+\beta_1 \sigma_{t-1}^2,\;\; t\in\ZZ,
  \]
  where $\alpha_0,\alpha_1,\beta_1>0,$ and $\sequence{Z}$ is a sequence of i.i.d. random
  variables with mean zero and variance one. Under mild conditions the process $\{X_t\}$ is
  regularly varying and satisfies Assumptions \ref{hypo:AC} and \ref{hypo:Aprimecluster}.  These
  hold for instance in the case of standard normal innovations ${Z}_t$ and sufficiently small
  parameters $\alpha_1, \beta_1$, see \cite[Section~4.4]{mikosch:wintenberger:2013}. Consider for
  simplicity such a stationary GARCH$(1,1)$ process with tail index $\alpha\in (0,1).$ Since all the
  conditions of \Cref{thm:PartialSumConvinE} are met, its partial sum process has a limit in the
  space $E$ (cf. Figure~\ref{FigGARCH}).
\end{example}

\begin{figure}[h]
  \centering
  \includegraphics[width=0.8\linewidth]{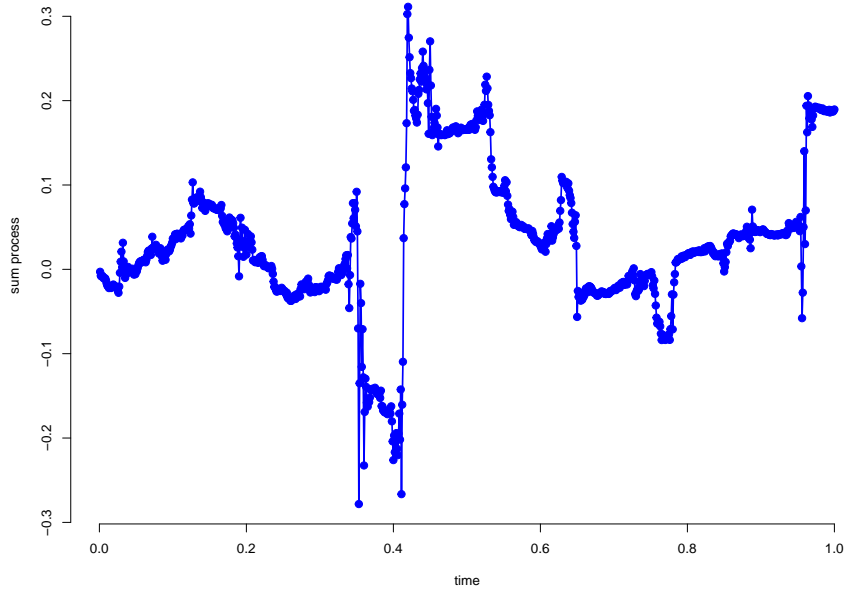}
  \caption{A simulated sample path of the process $S_n$ in the case of GARCH$(1,1)$ process with
    parameters $\alpha_0=0.01,\, \alpha_1=1.45$ and $\beta_1=0.1,$ and tail index $\alpha$ between
    $0.5$ and 1.}
  \label{FigGARCH}
\end{figure}
\begin{proof}[Proof of \Cref{thm:PartialSumConvinE}]
  The proof is split into the cases $\alpha<1$ which is simpler, and the case
  $\alpha\in[1,2)$ where centering and truncation introduce additional technical difficulties. 
  \begin{enumerate}[(a),wide=0pt]
  \item \label{item:01} Assume first that $\alpha\in(0,1).$ We divide the proof into several steps.
    \begin{enumerate}[{Step} 1.,wide=0pt]
    \item \label{item:step1} For every $\epsilon>0$, consider the functions $s^\epsilon,u^\epsilon$
      and $v^\epsilon$ defined on $\lo$ by
      \[
      s^\epsilon (\tbx) = \sum_j x_j \1{\{|x_j|>\epsilon\}} \; , \ \ u^\epsilon (\tbx) = \inf_k
      \sum_{j\leq k} x_j \1{\{|x_j|>\epsilon\}} \; , \ \ v^\epsilon (\tbx) = \sup_k \sum_{j\leq k}
      x_j \1{\{|x_j|>\epsilon\}} \; ,
      \]
      and define the mapping $T^\epsilon:\mathcal{M}_p([0,1]\times\lo\setminus\{\tilde{\boldsymbol{0}}\})
      \rightarrow E$ by setting, for $\pointmeas=\sum_{i=1}^\infty \delta_{t_i,\tbx^i},$
      \begin{align*}
        T^\epsilon \pointmeas= \left(\left(\sum_{t_i\leq t} s^\epsilon (\tbx^i)\right)_{t\in [0,1]} \;,\{t_i
          : \|\tbx^i\|_\infty >\epsilon \} , \{I(t_i):\|\tbx^i\|_\infty >\epsilon\} \right),
      \end{align*}
      where
      \[
I(t_i)=\sum_{t_j<t_i} s^\epsilon (\tbx^j) +\left[
      \sum_{t_k=t_i} u^\epsilon (\tbx^k) , \sum_{t_k=t_i} v^\epsilon (\tbx^k)\right].
\]
Since $m$ belongs $\mathcal{M}_p([0,1]\times\lo\setminus\{\tilde{\boldsymbol{0}}\}),$ there is only a finite number of points
$(t_i,\tbx^i)$ such that $\|\tbx^i\|_\infty >\epsilon$ and furthermore, every $\tbx^i$ has at most
finitely many coordinates greater than $\epsilon.$ Therefore, the mapping $T^\epsilon$ is
well-defined, that is, $T^\epsilon \pointmeas$ is a proper element of $E.$

Next, we define the subsets of $\mathcal{M}_p([0,1]\times\lo\setminus\{\tilde{\boldsymbol{0}}\})$
by
\begin{align*}
  \Lambda_1 & = \left\{\sum_{i=1}^\infty \delta_{t_i,\tbx^i}: |x_j^i|\neq \epsilon, i\geq 1, j\in\ZZ\right\} \; , \\
  \Lambda_2 & =\{\sum_{i=1}^\infty \delta_{t_i,\tbx^i}: 0<t_i<1 \textnormal{ and } t_i\neq t_j
  \textnormal{ for every } i>j\geq 1\}.
\end{align*}
We claim that $T^\epsilon$ is continuous on the set $\Lambda_1 \cap \Lambda_2.$ Assume that
$\pointmeas_n \rightarrow_{w^\#} \pointmeas=\sum_{i=1}^\infty \delta_{t_i,\tbx^i} \in\Lambda_1 \cap
\Lambda_2$.
By an adaptation of \cite[Proposition 3.13]{resnick:1987}, this convergence implies that the
finitely many points of $\pointmeas_n$ in every set $B$ bounded for $\tilde{d}'$ and such that
$\gamma(\partial B)=0$ converge pointwise to the finitely many points of $\pointmeas$ in $B$. In
particular, this holds for $B=\{(t,\tbx)\, : \, \|\tbx\|_\infty >\epsilon \}$ and it follows that
for all $t_1<t_2$ in $[0,1]$ such that
$\pointmeas(\{t_1,t_2\}\times \lo\setminus\{\tilde{\boldsymbol{0}}\})=0$,
\[
M_{t_1,t_2}(T^\epsilon(\pointmeas_n))\rightarrow M_{t_1,t_2}(T^\epsilon(\pointmeas)) \text{   in   } \RR
\]
and
\[
M_{t_1,t_2}(-T^\epsilon(\pointmeas_n))\rightarrow M_{t_1,t_2}(-T^\epsilon(\pointmeas)) \text{   in   } \RR,
\]
whith the local-maximum function $M_{t_1,t_2}$ defined as in (\ref{e:locmax}).  Since the set of all
such times is dense in $[0,1]$ and includes $0$ and $1$, an application of \Cref{thm:charofconvE}
gives that
\begin{align*}
  T^\epsilon(\pointmeas_n) \rightarrow T^\epsilon(\pointmeas)
\end{align*}
in $E$ endowed with the $M_2$ topology.

Recall the point process $\PPC = \sum_{i=1}^\infty \delta_{(T_i, P_i\bsQ_{i})}$ from
\eqref{eq:Repr-sum}. Since the mean measure of $\sum_{i=1}^\infty\delta_{(T_i,P_i)}$ does not have
atoms, it is clear that $N''\in \Lambda_1\cap \Lambda_2$~a.s. Therefore, by the convergence $\PPC_n\dto \PPC$ and the
continuous mapping argument
$$
\tilde{S}_{n,\epsilon}'\dto S_\epsilon' \; ,
$$
where $\tilde{S}_{n,\epsilon}'=T^\epsilon(N_n'')$ and $S_\epsilon'=T^\epsilon(N'').$
\item \label{item:step2} Recall that $W_i=\sum_{j\in \ZZ}|Q_{i,j}|$ and
  $\sum_{i=1}^\infty\delta_{P_i W_i}$ is a Poisson point process on $(0,\infty]$ with intensity
  measure $\theta \EE [W_i^\alpha] \alpha y^{-\alpha -1}dy$ (see \Cref{rem:useful-W}). Since
  $\alpha<1$,  one can sum up the points $\{P_iW_i\},$ i.e. 
\begin{equation}
  \label{e:Sum}
  \sum_{i=1}^\infty P_i W_i  = \sum_{i=1}^\infty \sum_{j\in \ZZ}P_i |Q_{i,j}|<\infty \; \text{ a.s.}
\end{equation}
Therefore, defining $s (\tbx) = \sum_j x_j$, we obtain that the process
\[
V(t)=\sum_{T_i\leq t} s(P_i \bsQ_{i}) \; , \; t\in [0,1] \; ,
\]
is almost surely a well-defined element in $D$ and moreover, it is an $\alpha$-stable L\'evy process.
Further, we define an element $V'$ in $E([0,1],\RR)$ by
\begin{align}
  V'= \left(V , \{T_i \}_{i\in\NN} , \{I(T_i)\}_{i\in \NN} \right) \; ,\label{eq:limitS1} 
\end{align}
where
\begin{align*}
  I(T_i) & = V(T_i-) +\left[ u(P_i \bsQ_{i}) , v(P_i \bsQ_{i})\right] \; , \\
  u(\tbx) & = \inf_k \sum_{j\leq k} x_j \; , \ \ v(\tbx) = \sup_k \sum_{j\leq k} x_j \; .  
\end{align*}
Since for every $\delta>0$ there are at most finitely many points $P_i W_i$ such that $P_i W_i
>\delta$ and $diam(I(T_i))=v(P_i \bsQ_{i})-u(P_i \bsQ_{i})\leq P_i W_i,$ $V'$ is indeed a proper
element of $E$ a.s.

We now show that, as $\epsilon\rightarrow 0,$ the limits $S_\epsilon'$ from the previous step
converge to $V'$ in $(E,m)$ almost surely.  Recall the uniform metric $m^*$ on $E$ defined in
(\ref{e:unmetric1}). By (\ref{e:Sum}) and dominated convergence theorem
\begin{equation}
  \label{eq:mbound1}
  m^*(S_\epsilon',V')=\sup_{0\leq t \leq 1} m(S_\epsilon'(t),V'(t))\leq \sum_{i=1}^\infty \sum_{j\in \ZZ}P_i |Q_{ij}|\1{\{P_i |Q_{ij}| \leq
    \epsilon\}}\rightarrow 0 \; ,  
\end{equation}
almost surely as $\epsilon\rightarrow 0$. Indeed, let $S_\epsilon$ be the \cadlag\ part of
$S_\epsilon',$ i.e.
\[
S_\epsilon(t)=\sum_{T_i\leq t}s^\epsilon(P_i \bsQ_i)=\sum_{T_i\leq t}\sum_{j\in\ZZ}P_i Q_{i,j}\1{\{|P_iQ_{i,j}|>\epsilon\}} \;,\; t\in [0,1]\; .
\] 
If $t\notin\{T_i\}$ then 
\[
m(S_\epsilon'(t),V'(t))=|S_\epsilon(t)-V(t)|\leq \sum_{T_i\leq t}\sum_{j\in\ZZ}P_i|Q_{i,j}|\1{\{P_i |Q_{ij}| \leq
    \epsilon\}}.
\] 
Further, when $t=T_k$ for some $k\in\ZZ,$ by using (\ref{eq:HasusdorffforInterval}) we obtain
\begin{equation}\label{eq:mbound2}
\begin{aligned}
m(S_\epsilon'(t),V'(t))&\leq |(S_\epsilon(T_k-)+v^\epsilon(P_k\bsQ_k))-(V(T_k-)+v(P_k\bsQ_k)) | \\
& \vee |(S_\epsilon(T_k-)+u^\epsilon(P_k\bsQ_k))-(V(T_k-)+u(P_k\bsQ_k)) |.  
\end{aligned}
\end{equation}
The first term on the right-hand side of the equation above  is bounded by
\begin{align*}
&\left|\sum_{T_i< T_k}\sum_{j\in\ZZ}P_i Q_{i,j}\1{\{P_i |Q_{ij}| \leq
    \epsilon\}}\right|+\left|\sup_{l\in\ZZ} \sum_{j\leq l}P_k Q_{k,j}\1{\{P_i |Q_{i,j}| >
    \epsilon\}}-  \sup_{l\in\ZZ} \sum_{j\leq l}P_k Q_{k,j}\right| \\
    &\leq \left|\sum_{T_i< T_k}\sum_{j\in\ZZ}P_i Q_{i,j}\1{\{P_i |Q_{ij}| \leq
    \epsilon\}}\right|+\sup_{l\in\ZZ}\left| \sum_{j\leq l}P_k Q_{k,j}\1{\{P_i |Q_{k,j}| >
    \epsilon\}}-  \sum_{j\leq l}P_k Q_{k,j}\right| \\
    &\leq \sum_{T_i\leq T_k}\sum_{j\in\ZZ}P_i|Q_{i,j}|\1{\{P_i |Q_{i,j}|\leq \epsilon\}}.
\end{align*}
Since, by similar arguments, one can obtain the same bound for the second term on the right-hand
side of (\ref{eq:mbound2}), (\ref{eq:mbound1}) holds.
It now follows from (\ref{e:unmetric2}) that  
\[
S_\epsilon' \rightarrow V' 
\] 
almost surely in $(E,m).$
\item \label{item:step3}Recall that
\[
\bsX_{n,i} =(X_{(i-1)r_n+1},\dots,X_{ir_n})/a_n
\] 
for $i=1,\dots,k_n$ and let $\tilde{S}_n'$ be an element in $E$ defined by 
  \begin{align*}
    \left(\left(\sum_{i/k_n \leq t} s(\bsX_{n,i})\right)_{t\in [0,1]} \;, \{i/k_n \}_{i=1}^{k_n} , \{I(i/k_n)\}_{i=1}^{k_n} \right)\; ,
   \end{align*} 
where 
\[
I(i/k_n)=\sum_{j<i} s(\bsX_{n,j}) +\left[
      u(\bsX_{n,i}) , v(\bsX_{n,i})\right] \; .
\]
By \cite[Theorem~4.2]{billingsley:1968} and the previous two steps, to show that 
\[
\tilde{S}_n'\dto V'
\]
in $(E,\metricE)$, it suffices to prove that, for all $\delta>0,$
\begin{align}
    \lim_{\epsilon\rightarrow 0}\limsup_{n\rightarrow
      \infty} \PP(\metricE(\tilde{S}_{n,\epsilon}',\tilde{S}_n')>\delta)=0 \; . \label{eq:tightnessE}
  \end{align}
Note first that, by the same arguments as in the previous step, we have   
\begin{align*}
  \unifE(\tilde{S}_{n,\epsilon}',\tilde{S}_n')\leq \sum_{j=1}^{k_nr_n}\frac{|X_j|}{a_n}\1{\{|X_j|\leq
    a_n\epsilon\}} \; .
\end{align*}
By (\ref{e:unmetric2}), Markov's inequality and Karamata's theorem
(\cite[Proposition~1.5.8]{bingham:goldie:teugels:1989})
\begin{align*}
  \limsup_{n\to\infty} \PP(\metricE(\tilde{S}_{n,\epsilon}',\tilde{S}_n')>\delta) & \leq
  \limsup_{n\to\infty}  \frac{k_nr_n}{\delta a_n} \EE\left[|X_1|\1{\{|X_1|\leq a_n\epsilon\}}\right]  \\
  & = \lim_{n\to\infty}  \frac{n}{\delta a_n} \cdot \frac{\alpha a_n\epsilon\PP(|X_1|>a_n\epsilon)}{1-\alpha}
   =  \frac{\alpha}{(1-\alpha)\delta}\epsilon^{1-\alpha} \; .
\end{align*}
This proves~(\ref{eq:tightnessE}) since $\alpha<1$ and hence 
\[
\tilde{S}_n'\dto V'
\]
in $(E,\metricE)$.
\item \label{item:step4} Finally, to show that the original partial sum process $S_n$ (and therefore
  $V_n$ since $\alpha\in(0,1)$) also converges in distribution to $V'$ in $(E,\metricE)$, by Slutsky
  argument it suffices to prove that
\begin{equation}
  \label{eq:AsympEqE1}
  \metricE(S_n,\tilde{S}_n')\Pto 0 \; .
\end{equation}
Recall that $k_n=\lfloor n / r_{n} \rfloor$ so $\frac{i r_n}{n}\leq \frac{i}{k_n}$ for all
$i=0,1,\dots,k_n$ and moreover
\begin{equation}
  \label{eq:rn/n}
  \frac{i}{k_n}-\frac{i r_n}{n} = 
  \frac{i}{k_n}\left(1-\frac{k_n r_n}{n}\right) \leq 1-\frac{\lfloor n / r_{n} \rfloor }{n/r_n} 
  = 1-\left(1-\frac{\{ n / r_{n}\}}{n/r_n}\right)\leq \frac{r_n}{n} \; .
\end{equation}
Let $d_{n,i}$ for $i=0,\dots,k_n-1$ be the Hausdorff distance between restrictions of graphs
$\Gamma_{S_n}$ and $\Gamma_{\tilde{S}_n'}$ on time intervals $(\frac{i r_n}{n},\frac{(i+1)r_n}{n}]$
and $(\frac{i}{k_n},\frac{i+1}{k_n}],$ respectively (see Figure \ref{fig:M2equiv}).
 
\begin{figure}[h]
  \centering
  \includegraphics[width=.45\linewidth]{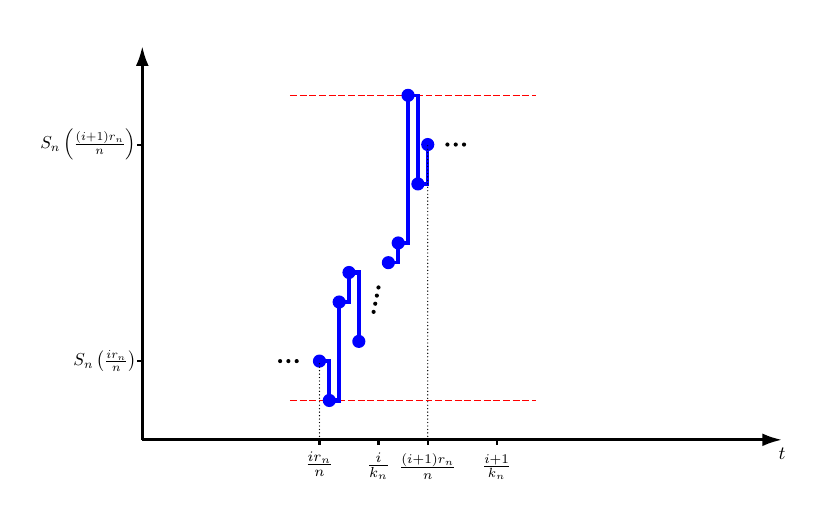}
 \includegraphics[width=.45\linewidth]{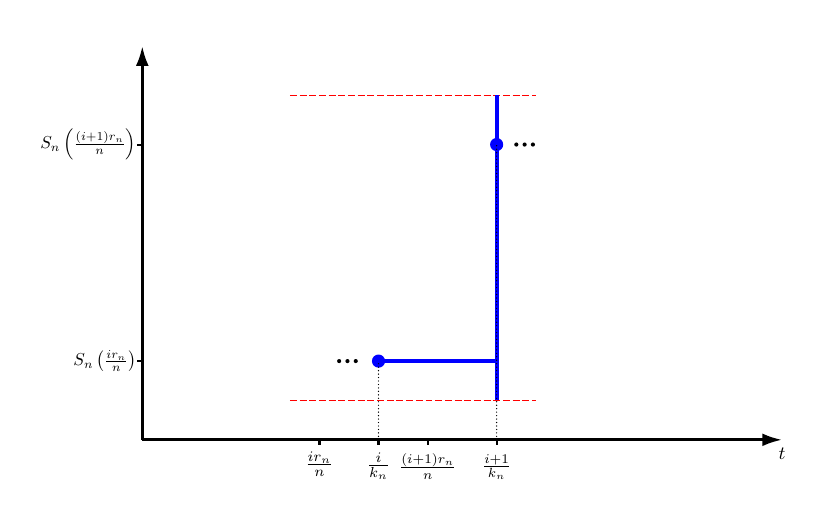}
  \caption{Restrictions of graphs $\Gamma_{S_n}$ and $\Gamma_{\tilde{S}_n'}$ on time intervals
    $(\frac{i r_n}{n},\frac{(i+1)r_n}{n}]$ and $(\frac{i}{k_n},\frac{i+1}{k_n}],$ respectively.}
  \label{fig:M2equiv}
\end{figure}

First note that, by (\ref{eq:rn/n}), the time distance between any two points on these graphs is at
most $2r_n/n.$ Further, by construction, $S_n$ and $\tilde{S}_n'$ have the same range of values on
these time intervals. More precisely,
\[
\bigcup_{t\in(\frac{i r_n}{n},\frac{(i+1)r_n}{n}]}\{z\in\RR\;:\;
(t,z)\in\Gamma_{S_n}\}=\bigcup_{t\in(\frac{i}{k_n},\frac{i+1}{k_n}]}\{z\in\RR\;:\;
(t,z)\in\Gamma_{\tilde{S}_n'}\}=\tilde{S}_n'((i+1)/k_n) \; .
\]
Therefore, the distance between the graphs comes only from the time component, i.e. 
\[
d_{n,i}\leq \frac{2 r_n}{n}\rightarrow 0\; , \; \text{ as } n\rightarrow \infty\; ,
\]
for all $i=0,1,\dots,k_n-1.$ 

Moreover, if we let $d_{n,k_n}$ be the Hausdorff distance between the restriction of the graph
$\Gamma_{S_n}$ on $(\frac{k_n r_n}{n},1]$ and the interval $(1,\tilde{S}_n'(1)),$ it holds that
\[
d_{n,k_n} \leq \frac{r_n}{n} \vee  \sum_{j=k_n r_n+1}^n\frac{|X_j|}{a_n}\Pto 0,
\]
as $n\rightarrow \infty.$
Hence, (\ref{eq:AsympEqE1}) holds since
\[
\metricE(S_n,\tilde{S}_n')\leq \bigvee_{i=0}^{k_n}d_{n,i},
\]
and this finishes the proof in the case $\alpha\in(0,1)$.

\end{enumerate}

\item Assume now that $\alpha\in [1,2)$.  As shown in~\ref{item:step1} in the proof of
  \ref{item:01}, it holds that
  \begin{align}
    \tilde{S}_{n,\epsilon}'\dto S_\epsilon' \label{eq:conv-epsilon}
  \end{align}
  in $E,$ where $\tilde{S}_{n,\epsilon}'=T^\epsilon(N_n'')$ and $S_\epsilon'=T^\epsilon(N'')$.

  For every $\epsilon>0$ define a \cadlag\ process $S_{n,\epsilon}$ by setting, for $t\in [0,1],$
  \[
  S_{n,\epsilon}(t)=\sum_{i=1}^{\lfloor nt \rfloor} \frac{X_i}{a_n}\1{\{|X_i|/a_n > \epsilon\}}.
  \]
  Using the same arguments as in \ref{item:step4} in the proof of \ref{item:01}, it holds that, as
  $n\rightarrow\infty,$
  \begin{align}
    \metricE(S_{n,\epsilon},\tilde{S}_{n,\epsilon}')\Pto 0 \; .  \label{eq:slutskyargument}
  \end{align}
  Therefore, by Slutsky argument it follows from (\ref{eq:slutskyargument}) and~(\ref{eq:conv-epsilon}) that
  \begin{equation}
    \label{eq:epsilonConv}
    S_{n,\epsilon}\dto S_\epsilon'
  \end{equation}
  in $(E,\metricE)$.

  Since $\alpha\in [1,2)$ we need to introduce centering, so we define the c\`adl\`ag process
  $V_{n,\epsilon}$ by setting, for $t\in[0,1],$
  \[
  V_{n,\epsilon}(t)=S_{n,\epsilon}(t)-\lfloor nt \rfloor
  \EE\left(\frac{X_1}{a_n}\1{\{\epsilon<|X_1|/a_n\leq 1\}}\right)\;.
  \]
  From (\ref{eq:conv-mu}) we have, for any $t \in [0,1],$ as $n\rightarrow \infty,$
\begin{equation}
  \label{eq:ConvOfCent}
  \lfloor nt \rfloor \EE\left(\frac{X_1}{a_n}\1{\{\epsilon< |X_1|/a_n \leq 1\}}\right)\rightarrow t \int_{\{x\; : \;\epsilon < |x| \leq 1\}} x \mu(dx).
\end{equation}
Since the limit function above is continuous and the convergence is uniform on $[0,1]$, by Lemma
\ref{lem:ContinuityOfAdditionInE} and (\ref{eq:epsilonConv}) it follows that
\begin{equation}
  \label{eq:convergenceVneps}
  V_{n,\epsilon}\dto V_\epsilon '
\end{equation}
in $E$, where $V_\epsilon'$ is given by (see \Cref{rem:additionInE})
\begin{align*}
V_\epsilon'(t) & = S_\epsilon'(t) -t \int_{\{x\; : \;\epsilon < |x| \leq 1\}} x \mu(dx)  \; .
\end{align*}
Let $V_\epsilon$ be the c\`adl\`ag part of $V_\epsilon',$ i.e., 
\begin{align}
  V_\epsilon(t) & = \sum_{T_i\leq t} s^\epsilon(P_i \bsQ_i) - t \int_{\{x\; : \;\epsilon < |x| \leq 1\}} x \mu(dx) \; .
  \label{eq:Vepsilon}
\end{align}
By \Cref{lem:newlem}, there exist an $\alpha$-stable L\'evy process $V$ such that $V_\epsilon$
converges uniformly almost surely (along some subsequence) to $V$.  Next, as in \ref{item:step2} in
the proof of \ref{item:01} we can define a element $V'$ in $E$ by
\begin{equation}
  \label{eq:limitV1}
  V'=\left(V,\{T_i\}_{i\in \NN}, \{I(T_i)\}_{i\in \NN}\right) \; , 
\end{equation}
where
\begin{equation}
  \label{eq:limitV2}
  I(T_i)=V(T_i-) +\left[    u(P_i \bsQ_{i}) , v(P_i \bsQ_{i})\right] \; ,
\end{equation}
Also, one can argue similarly to the proof of (\ref{eq:mbound1}) to conclude that
\[
\unifE(V',V_\epsilon')\leq ||V-V_\epsilon||_\infty + \sup_{i\in \NN} \sum_{j\in \ZZ}
P_i|Q_{i,j}|\1{\{P_i|Q_{i,j}|\leq \epsilon\}} \; .
\]
Now it follows from (\ref{e:unmetric2}), (\ref{eq:negligibilityCond1}) and a.s. uniform convergence
of $V_\epsilon$ to $V$ that
\begin{equation}
  \label{eq:convergenceVeps}
  V_\epsilon'\rightarrow V' \; \text{a.s.}
\end{equation}
in $(E,\metricE)$. 

Finally, by (\ref{eq:convergenceVneps}), (\ref{eq:convergenceVeps}) and
\cite[Theorem~4.2]{billingsley:1968}, to show that
\begin{equation}
  \label{eq:convergenceVn}
  V_n\dto V',
\end{equation}
in $(E,\metricE)$, it suffices to prove that, for all $\delta>0,$
\begin{align}
  \lim_{\epsilon\rightarrow 0}\limsup_{n\rightarrow \infty} \PP(\metricE(V_{n,\epsilon},V_n)>\delta)=0 \;
  . \label{eq:tightnessE2}
\end{align}
But this follows from \Cref{hypo:ANSJ} and (\ref{e:unmetric2}) since
\[
\unifE(V_{n,\epsilon},V_n) \leq \max_{1\leq k \leq n} \left|\sum_{i=1}^k\left(\frac{X_i}{a_n}
    \1{\{|X_i|/a_n\leq \epsilon\}} - \EE\left[\frac{X_i}{a_n}\1{\{|X_i|/a_n\leq \epsilon\}}\right]
  \right)\right|.
\]
Hence (\ref{eq:convergenceVn}) holds and this finishes the proof.
\end{enumerate}

\end{proof}

\subsection{Supremum of the partial sum process}
\label{sec:supremum-partial-sum}

We next show that the supremum of the partial sum process converges in distribution in $D$ endowed
with the $M_1$ topology, where the limit is the ``running supremum'' of the limit process $V'$ from
Theorem \ref{thm:PartialSumConvinE}.

Let $V$ be the L\'evy process defined in (\ref{eq:levylimit}) and define the process $V^+$ on $[0,1]$ by
\begin{align*}
  V^+(t)= &\begin{cases} V(t)\;, & t\notin\{T_j\}_{j\in\NN}\\
                        V(t-)+\sup_{k\in\ZZ}\sum_{j\leq k}P_i Q_{i,j} \; ,& t=T_i \text{ for some } i\in\NN.\ \
                        \end{cases}
\end{align*}
Define $V^-$ analogously using infimum instead of supremum.  Note that $V^+$ and $V^-$ need not be
right-continuous at the jump times $T_j.$ However, their partial supremum or infimum are \cadlag\
functions.

\begin{theorem}
  \label{thm:SupPartSumConv}
  Under the same conditions as in \Cref{thm:PartialSumConvinE}, it holds that
  \[
  \left(\sup_{s\leq t} V_n (s)\right)_{t\in [0,1]} \dto \left(\sup_{s\leq t} V^+ (s)\right)_{t\in
    [0,1]} \; , 
  \]
  and
  \[
  \left(\inf_{s\leq t} V_n (s)\right)_{t\in [0,1]} \dto \left(\inf_{s\leq t} V^- (s)\right)_{t\in
    [0,1]} \; , 
  \]
jointly   in $D([0,1],\RR)$ endowed with the $M_1$ topology.
\end{theorem}

\begin{proof}
  We prove the result only for the supremum of the partial sum process since the infimum case is
  completely analogous and joint convergence holds since we are applying the continuous mapping
  argument to the same process.

Define the mapping $\;\sup:E([0,1],\RR)\rightarrow D([0,1],\RR)$ by
  \begin{align*}
    \sup(x')(t)=\sup\{z:z\in x'(s), 0\leq s \leq t\} \; .
  \end{align*}
  Note that $\sup(x')$ is non-decreasing and since for every $\delta>0$ there
  are at most finitely many times $t$ for which the $diam(x'(t))$ is greater than $\delta$, by
  \cite[Theorem 15.4.1.]{whitt:2002} it follows easily that this mapping is well-defined, i.e. that
  $\sup(x')$ is indeed an element in $D.$
Also, by construction,
\[
\sup(V')=\left(\sup_{s\leq t} V^+ (s)\right)_{t\in [0,1]}  
\]
and 
\[
\sup(V_n)=\left(\sup_{s\leq t} V_n (s)\right)_{t\in [0,1]}.
\]
Define the subset of $E$ by
\[
\Lambda=\left\{x'\in E \: : \: x'(0)=\{0\}\right\}
\]
and assume that $x_n'\rightarrow x'$ in $(E,\metricE),$ where $x_n',x'\in \Lambda$. By Theorem
\ref{thm:charofconvE} it follows that
\[
\sup(x_n')(t)=M_{0,t}(x_n')\rightarrow M_{0,t}(x')=\sup(x')(t)
\]
for all $t$ in a dense subset of (0,1], including $1$. Also, the convergence trivially holds for
$t=0$ since $\sup(x_n')(0)=\sup(x')(0)=0$ for all $n\in \NN.$ Since $\sup(x')$ is non-decreasing for
all $x'\in E,$ we can apply \cite[Corollary~12.5.1]{whitt:2002} and conclude that
\begin{align*}
  \sup(x_n') \rightarrow \sup(x')
\end{align*}
in $D$ endowed with $M_1$ topology. 
  Since $V_n,V'\in \Lambda$
almost surely, by Theorem \ref{thm:PartialSumConvinE} and continuous mapping argument it follows
that
\[
\left(\sup_{s\leq t} V_n (s)\right)_{t\in [0,1]} \dto \left(\sup_{s\leq t} V^+ (s)\right)_{t\in [0,1]} 
\] 
in $D$ endowed with $M_1$ topology. 
\end{proof}

\begin{remark}
  Note that when $\sum_{j\in\ZZ}Q_{j}=0$ a.s., the limit for the supremum of the partial sum process in Theorem~\ref{thm:SupPartSumConv} is simply a so
  called Fr\'echet extremal process.  For an illustration of the general limiting behavior of
  running maxima in the case of a linear processes, consider again the moving average of order 1
  from Example~\ref{ex:MAinf4}. Figure~\ref{FigMA1} shows a path (dashed line) of the running maxima
  of the MA$(1)$ process $X_t = \xi_{t} -0.7 \xi_{t-1}$.
\end{remark}

\subsection{$M_2$ convergence of the partial sum process}
We can now characterize the convergence of the partial sum process in the $M_2$ topology in
$D([0,1])$ by an appropriate condition on the tail process of the sequence $\sequence{X}$.
\begin{hypothesis}
  \label{hypo:PartialSumConvM2}
  The sequence $\sequence{Q}[\ZZ][j]$ satisfies
    \begin{align}
      -\left( \sum_{j\in\ZZ} Q_{j} \right)_- = \inf_{k\in\ZZ} \sum_{j\leq k} Q_{j} \leq
      \sup_{k\in\ZZ} \sum_{j\leq k} Q_{j} = \left(\sum_{j\in\ZZ} Q_{j} \right)_+ \; \mbox{
        a.s. }
        \label{eq:convM2}
  \end{align}
 i.e. $-s(\bsQ)_- = u(\bsQ) \leq v(\bsQ) = s(\bsQ)_+$ a.s.
\end{hypothesis}
Note that this assumption ensures that $\sum_{j\in\ZZ} Q_j\ne0$ and that the limit process $V'$ from
\Cref{thm:PartialSumConvinE} has sample paths in the subset $D'$ of $E$ which was defined in
(\ref{D'}). By Lemma \ref{lem:homeomorphismD}, Theorem \ref{thm:PartialSumConvinE} and the
continuous mapping theorem, the next result follows immediately.

\begin{theorem}
  \label{theo:M2-partial-sum}
  If, in addition to conditions in \Cref{thm:PartialSumConvinE}, \Cref{hypo:PartialSumConvM2} holds,
  then
  \[
  V_n\dto V
  \]
  in $D([0,1],\RR)$ endowed with the $M_2$ topology.
\end{theorem}

Since the supremum functional is continuous with respect to the $M_2$ topology, this result implies
that the limit of the running supremum of the partial sum process is the running supremum of the
limiting $\alpha$-stable L\'evy process as in the case of \iid\ random variables. 

\begin{example}
  For the linear process
  $ X_t = \sum_{j\in\ZZ} c_j \xi_{t-j}$ from \Cref{subsec:linear} and \Cref{ex:MAinf4}, the
  corresponding sequence $\{Q_j\}$ was given in \eqref{eq:Qlin}. It follows that the Condition
  (\ref{eq:convM2}) can be expressed as
  \begin{align}
    - \left(\sum_{j\in\ZZ} c_j\right)_- = \inf_{k\in\ZZ} \sum_{j\leq k} c_j \leq \sup_{k\in\ZZ}
    \sum_{j\leq k} c_j = \left(\sum_{j\in\ZZ} c_j\right)_+  \; .  \label{eq:condition-BK}
  \end{align}
  This is exactly \cite[Condition~3.2]{basrak:krizmanic:2014}. Note that~(\ref{eq:condition-BK})
  implies that
  \begin{align*}
    \left|\sum_{j\in\ZZ} c_j \right| > 0 \; .
  \end{align*}
\end{example}

\section{Record times}
\label{sec:records}

In this section we study record times in a stationary sequence $\sequence{X}$.  Since 
record times remain unaltered after a strictly increasing transformation, the main result below
holds for stationary sequences with a general marginal distribution as long as they can be
monotonically transformed into a regularly varying sequence.

We start by introducing the notion of records for sequences in $\lo$.  
For  $y\geq 0$  and $\bx= \{x_j\} \in\lo$ define 
\[
R^{\bx} (y) = \sum_{j=-\infty}^\infty  \1{\{x_j > y \vee \sup_{i<j} x_i  \}}\,,
\]
representing the number of records in the sequence $\bx$ larger than $y$, which is finite for $\bx
\in \lo$. For notational simplicity, we suppress notation $\tbx$ in this section.

Let
$\pointmeas = \sum_{i=1}^\infty \delta_{t_i,\bx^i} \in \mcm_p({[0,\infty)
  \times\lo\setminus\{\tilde{\boldsymbol{0}}\}})$, where $\bx^i = \{x^i_j\} \in \lo$. Define, for $t>0$
\begin{align*}
  M^{\pointmeas}(t) = \sup_{t_i \leq t} \|\bx^i\|_\infty \; , \quad  M^{\pointmeas}(t-)  =  \sup_{t_i <  t}  \|\bx^i\|_\infty  \; , 
\end{align*}
where we set $\sup \emptyset = 0$ for convenience. Next, let $R_\pointmeas$ be the (counting) point
process on $(0,\infty)$ defined by
\begin{align*}
R_\pointmeas  & = \sum_{i} \delta_{t_i}  R^{\bx^i} (M^\pointmeas(t_i-))  \; , 
\end{align*}
hence for arbitrary  $0<a < b $ 
\begin{align*}
R_\pointmeas (a,b] & = \sum_{a< t_i \leq b} \sum_{j=-\infty}^\infty  \1{\{x^i_j > M^\pointmeas(t_i-) \vee \sup_{k<j} x^i_k  \}}\; .
\end{align*}

Consider the following subset of $\mcm_p({[0,\infty)\times\lo\setminus\{\tilde{\boldsymbol{0}}\}})$
\begin{align*}
  A = & \{ \pointmeas = \sum_{i=1}^\infty \delta_{t_i,\bx^i} : \mbox{ such that }  M^\pointmeas(t) >0
  \mbox{ for all $t>0$, while } \\
  & \mbox{all $t_i$'s are mutually different as well as all nonzero $x^i_j$'s} \}\,.
\end{align*}

The space $\mcm_p((0,\infty))$ is endowed with the $w^\#$ topology which is equivalent to the usual
vague topology since $(0,\infty)$ is locally compact and separable.
\begin{lemma}
  \label{Lm:RecCont}
  \label{lem:continuity-record}
  The mapping $\gamma\mapsto R_\gamma$ from $\mcm_p({[0,\infty)\times\lo\setminus\{\tilde{\boldsymbol{0}}\}})$ to $\mcm_p((0,\infty))$
  is continuous at every $\pointmeas \in A$.
\end{lemma}

\begin{proof}
  Fix an arbitrary $\pointmeas \in A$, and assume a sequence $\{\pointmeas_n\}$ in
  $ \mcm_p([0,\infty)\times\lo\setminus\{\tilde{\boldsymbol{0}}\})$ satisfies $\pointmeas_n \rightarrow_{w^\#} \pointmeas$.  We must
  prove that $R_{\gamma_n}$ converges vaguely to $R_\gamma$ in $\mcm_p((0,\infty))$.  By the
  Portmanteau theorem, it is sufficient to show that that for all $0<a< b \in \{ t_i\}^c$
  \[
  R_{\pointmeas_n} (a,b] \to R_{\pointmeas} (a,b]\,,
  \]
  as $n\to\infty$.  For $0<a< b \in \{ t_i\}^c$ there are finitely many time instances, say
  $t_{i_1},\ldots,t_{i_k}\in(a,b]$ such that $\|\bx^{i_l}\|_{\infty}> M^\pointmeas(a)>0$, hence
  \begin{align}
    \label{dRm}
    R_\pointmeas(a,b] & = \sum_{t_{i} \in (a,b]} R^{\bx^i} (M^{\pointmeas}(t_i-)) = \sum_{l=1}^{k} R^{\bx^{i_l}}
    (M^{\pointmeas}(t_{i_l}-)) \; .
  \end{align}
  For all $ \pointmeas_n = \sum_{i=1}^\infty \delta_{t^n_i,\bx^{n,i}}$ with $n$ large enough, there
  also exist exactly $k$ (depending on $a$ and $b$) time instances
  $t^n_{i_1},\ldots,t^n_{i_k}\in(a,b]$ such that $\|\bx^{n,i_l}\|_{\infty}>
  M^\pointmeas(a)$.
  Moreover, they satisfy $\bx^{n,i_l} \to \bx^{i_l}$ and $t^n_{i_l} \to t_{i_l}$ for $l=1,\ldots k$
  as $n\to\infty$. Hence for $n$ large enough
  \begin{align}
    \label{dRmn}
    R_{\pointmeas_n}(a,b] & = \sum_{t^n_{i} \in (a,b]} R^{\bx^{n,i}} (M^{\pointmeas_n}(t^n_i-)) = \sum_{l=1}^{k}
    R^{\bx^{n,i_l}} (M^{\pointmeas_n}(t^n_{i_l}-)) \;.
  \end{align}
  Assume $y_n \to y>0 $ and $\bx^n \to \bx=\{x_j\}$ where the non zero $x_j$ are pairwise distinct
  and $x_j\not = y$ for all $j\in\ZZ$.  Then, it is straightforward to check that
  \[
  R^{\bx^n}(y_n) \to R^{\bx}(y)\,.
  \]
  Observe further that for the choice of $t_{i_l}$, $t^n_{i_l}$ we made above, it holds that
  $M^{\pointmeas_n} (t^n_{i_l}-) \to M^{\pointmeas} (t_{i_l}-)$ since $t_i$'s are all different.  Together with
  \eqref{dRm} and \eqref{dRmn} this yields
  \[
  R_{\pointmeas_n}(a,b] \to R_{\pointmeas}(a,b]
  \]
  as $n\to\infty$.
\end{proof}

Since we are only interested in records, for convenience we consider a nonnegative stationary
regularly varying sequence $\sequence{X}$.

Adopting the notation of \Cref{Subs:ppc}, we denote
\begin{align} 
  \label{e:N2p}
  \PPC_n = \sum_{i=1}^{\infty} \delta_{(i/k_n,\bsX_{n,j})} \; .
\end{align}
We will also need the point process  $N_n$ defined  by
\begin{align*}
  N_n = \sum_{i=1}^{\infty} \delta_{(i/n,X_{i}/a_n)} \; .
\end{align*} 
The process $N_n$ can be viewed as a point process
on $[0,\infty) \times \RR\setminus\{0\}$, but since $\RR$ can be embedded in $\lo$ (by identifying a
real number $x\ne0$ to a sequence with exactly one nonzero coordinate equal to $x$), in the sequel we
treat it as a process on the space $[0,\infty)\times\lo\setminus\{\tilde{\boldsymbol{0}}\}$.

As in the previous section, we will assume that 
\begin{align}
  \PPC_n\dto \PPC  = \sum_{i=1}^\infty\delta_{(T_i, P_i\bsQ_{i})} \; , \label{eq:Repr-rec}
\end{align}
as $n\to\infty$,
where $\PPC$ has the same form as in \eqref{eq:Repr-sum}, but on the space $\mcm_p({[0,\infty)\times\lo\setminus\{\tilde{\boldsymbol{0}}\}})$. 
 For $\beta$-mixing and linear processes, this
 convergence follows by direct extension of
results in \Cref{sec:poipro} from the state space $[0,1]\times\lo\setminus\{\tilde{\boldsymbol0}\}$  to
$[0,T]\times \lo\setminus\{\tilde{\boldsymbol0}\}$ for arbitrary  $T>0$.

\begin{theorem}
  \label{thm:RecordTimes}
  Let $\sequence{X}$ be a stationary regularly varying sequence with tail index $\alpha>0$. Assume
  that the convergence in~(\ref{eq:Repr-rec}) holds and moreover that 
  $$
  \PP ( \mbox{all nonzero } Q_{1,j}\mbox{'s are mutually different}) = 1 \; .
  $$  
  Then
  \[
  R_{N_n} \dto R_{\PPC}\,,
  \]
  in $\mcm_p((0,\infty))$. Moreover, the limiting process is a compound Poisson process with
  representation
  \[
  R_{\PPC} = \sum_{i\in\ZZ} \delta_{\tau_i} \kappa_i \,,
  \]
  where $\sum_{i\in\ZZ} \delta_{\tau_i}$ is a Poisson point process on $(0,\infty)$ with intensity
  measure $x^{-1}\rmd x$ and $\{\kappa_i\}$ is a sequence of \iid\ random variables independent of
  it with the same distribution as the integer valued random variable $R^{\bsQ_1} (1/\paretorec)$
  where $\paretorec$ is a Pareto random variable with tail index $\alpha$, independent of~$\bsQ_1$.
\end{theorem}

\begin{proof}
  Since they are constructed from the same sequence $X_1,X_2,\ldots$, the record values of the point
  process $\PPC_n $ in \eqref{e:N2p} correspond to the record values of the point process $N_n$.
Using (\ref{eq:Repr-rec}) and the additional assumption on the $Q$'s, by
  \Cref{Lm:RecCont} it follows that
  \begin{align}
    \label{RecConv1}
    R_{\PPC_n}\dto R_{\PPC}\,.
  \end{align}
  Note that record times $i/n$ of the process $N_n$ appear at slightly altered times $(\lfloor i/r_n
  \rfloor +1) /k_n$ in the process $\PPC_n$.  However, asymptotically the record times are very
  close. Indeed, take $f \in C_K^+(0,\infty)$, then $f$ has a support on the set of the form
  $(a,b]$, $0<a<b$, which can be enlarged slightly to a set $(a-\epsilon,b+\epsilon]$, where
  $0<a-\epsilon$ for a sufficiently small $\epsilon >0$. Clearly, $f$ is continuous on that set,
  even uniformly.  Note further that for any such $\epsilon$ there is an integer $n_0$ such that for
  $n\geq n_0 $, $i/n \in (a,b]$ implies $(\lfloor i/r_n \rfloor +1) /k_n \in
  (a-\epsilon,b+\epsilon]$, and vice versa. Moreover, by uniform continuity of the function $f$,
  $n_0$ can be chosen such that for $n \geq n_0$,
  \[
  \left| f\left( \frac in \right) - f\left( \frac{\lfloor i/r_n \rfloor +1}{k_n} \right) \right|
  \leq \epsilon \; .
  \]
  Consider now the difference between Laplace functionals of the point processes $R_{\PPC_n}$ and
  $R_{N_n}$ for a function $f\in C_K^+(0,\infty)$ as above. Since $\epsilon$ above can be made
  arbitrarily small, it follows that
  \[
  \left| \EE [\rme^{-f(R_{\PPC_n})}] - \EE [\rme^{-f(R_{N_n})}] \right|\to 0 \; ,
  \]
  which together with \eqref{RecConv1} yields the convergence statement of the theorem.
	
  Consider now a point measure
  $\pointmeas = \sum_{i=1}^\infty \delta_{t_i,\bx^i} \in \mcm_p({[0,\infty)\times\lo\setminus\{\tilde{\boldsymbol{0}}\}})$, but such
  that all $\bx^i$ have only nonnegative components and all $t_i$'s are mutually different.  We say
  that a point measure $\gamma$ has a record at time $t$ if
  $(t,\bx) \in \{(t_i,\bx^i) : i \geq 1 \}$ and
  $M^{\pointmeas}(t-) = \sup_{t_i < t} \|\bx^i\|_{\infty} < \|\bx\|_{\infty}$\,. Taking the order into account, at
  time $t$ we will see exactly $R^{\bx} (M^{\pointmeas}(t-)) $ records.  Similarly, a point measure
  $\eta = \sum_{i=1}^\infty \delta_{t_i,x_i} \in \mcm_p({[0,\infty)\times[0,\infty)})$ has a record at
  time $t$ with corresponding record value $x$, if $(t,x) \in \{(t_i,x_i) : i \geq 1 \}$ and
  $\eta([0,t)\times [x,\infty)) = 0$.

  To prove the representation of the limit, observe that
  $\PPC = \sum_{i=1}^\infty \delta_{(T_i, P_i\bsQ_{i})} $ has records at exactly the same time
  instances as the process $M_0 = \sum_{i=1}^\infty \delta_{(T_i, P_i)} $, since by the assumptions
  of the theorem and by definition of the sequences $\bsQ_i$, all of their components are in $[0,1]$
  with one of them being exactly equal to 1.  Because $M_0$ is a Poisson point process on
  $[0,\infty)\times(0,\infty]$ with intensity measure $\rmd x \times \rmd(-\theta y^{-\alpha})$, it
  has infinitely many points of in any set of the form $[a,b]\times[0,\epsilon]$ with $a<b$ and
  $\epsilon>0$. Hence, one can a.s. write the record times of $M_0$ as a double sided sequence
  $\tau_n,\ n \in \ZZ$, such that $\tau_n < \tau_{n+1}$ for each $n$.  Fix an arbitrary $s>0$, and
  assume without loss of generality that $\tau_1$ represents the first record time strictly greater
  than $s$, i.e.  $\tau_1= \inf\{\tau_i : \tau_i >s\}$.  Denote the corresponding successive record
  values by $U_n$; they clearly satisfy $U_n < U_{n+1}$ and $U_n = \sup_{T_i\leq \tau_n} P_i$ and
  $U_0 = \sup_{T_i\leq s} P_i$.  According to \cite[Proposition~4.9]{resnick:1987},
  $\sum_{n\in\ZZ}\delta_{\tau_n}$ is a Poisson point process with intensity $x^{-1}\rmd x$ on
  $(0,\infty)$.  Apply now \cite[Proposition~4.7~(iv)]{resnick:1987} (note that $U_0$ corresponds to
  $Y(s)$ in the notation of that proposition) to prove that $\{U_{n}/U_{n-1},n\geq 1\}$ is a
  sequence of \iid\ random variables with a Pareto distribution with tail index $\alpha$.  Because
  the record times $\tau_n$ and record values $U_n$ for $n\geq 1$ of the point process $M_0$ match
  the records of point process $\PPC = \sum_{i=1}^\infty \delta_{(T_i, P_i\bsQ_{i})}$ on the
  interval $(s,\infty)$, we just need to count how many of them appear at any give time $\tau_n$
  which are larger than the previous record $U_{n-1}$. If, say, $\tau_n=T_i$, that number
  corresponds to the number of $Q_{ij}$'s which after multiplication by the corresponding $U_n=P_i$
  represent a record larger than $U_{n-1}$.  Hence, that random number has the same distribution as
  \[
  \kappa = R^{\bsQ} (U_{1}/U_{0}) \; . 
  \]
  Recall that $s>0$ was arbitrary. Now since the point process
  $\sum_{i=1}^\infty \delta_{(T_i, P_i)} $ and therefore sequence $ \{U_n/U_{n-1}, n\geq 1 \}$ is
  independent of the \iid\ random elements $\{\bsQ_i\}$ and since $U_1/U_0$ has a Pareto
distribution with tail index $\alpha$, the claim follows.
\end{proof}

\begin{example} 
  For an illustration of the previous theorem, consider the moving average process of order
  $1$ 
  \begin{align*}
    X_t = \xi_{t} + c \xi_{t-1} \; ,
  \end{align*}
  for a sequence of \iid\ nonnegative random variables $\{\xi_t,t\in\ZZ\}$ with regularly varying
  distribution and the tail index $\alpha>0$. Assume further that $c>1$. By \eqref{eq:Qlin}, the sequence
  $\{Q_j\},$ as a random element in the space $\lo,$ is in this case equal to the deterministic sequence
  $\{\ldots, 0, 1/c,1,0,\ldots \}$. Intuitively speaking, in each cluster of extremely large values,
  there are exactly two successive extreme values with the second one $c$ times larger that the
  first. Therefore, each such cluster can give rise to at most $2$ records. By straightforward
  calculations, the random variables $\kappa_i$ from \Cref{thm:RecordTimes} have the following
  distribution
  \[
  \PP(\kappa_i=2)=\PP(1/\paretorec\leq 1/c)=\PP(\paretorec\geq c)=\frac{1}{c^\alpha}=1-\PP(\kappa_i=1) \; .
  \]
\end{example}

\section{Lemmas}
\label{sec:proof-trivia}

\subsection{Metric on the space $\lo$}
Let $(\mathbb{X},d)$ be a metric space, we define the distance between $x\in \mathbb{X}$ and subset $B\subset \mathbb{X}$ by
$d(x,B)=\inf\{d(x,y):y\in B\}.$ Let $\sim$ be an equivalence relation on $\mathbb{X}$ and let $\tilde{\mathbb{X}}$ be
the induced quotient space. Define a function $\dtilde:\tilde{\mathbb{X}}\times\tilde{\mathbb{X}}\rightarrow
[0,\infty)$ by:
  \begin{align*}
    \dtilde(\tilde{x},\tilde{y}) = \inf\{d(x',y'):x'\in\tilde{x},y'\in\tilde{y}\} \;, 
  \end{align*}
  for all $\tilde{x},\tilde{y}\in \tilde{\mathbb{X}}.$
\begin{lemma}
  \label{lem:quotient-complet}
  Let $(\mathbb{X},d)$ be a complete separable metric space. Assume that for all $\tilde{x},\tilde{y}\in\tilde{\mathbb{X}}$ and
  all $x,x'\in\tilde{x}$ we have
  \begin{align}
    d(x,\tilde{y})=d(x',\tilde{y}). \;
    \label{eq:condition-tilded}
  \end{align}
  Then $\dtilde$ is a pseudo-metric which makes $\tilde{\mathbb{X}}$ a separable and
  complete pseudo-metric space.
\end{lemma}

\begin{proof}
  To prove that $\dtilde$ is a pseudo-metric, the only nontrivial step is to show that $\dtilde$
  satisfies the triangle inequality, but that is implied by
  Condition~(\ref{eq:condition-tilded}). Separability is easy to check and it remains to prove that
  $(\tilde{\mathbb{X}},\dtilde)$ is complete.
  
  Let $\{\tilde{x}_n\}$ be a Cauchy sequence in $(\tilde{\mathbb{X}},\dtilde)$. Then we can find a
  strictly increasing sequence of nonnegative integers $\{n_k\}$ such that
  $$\dtilde(\tilde{x}_m,\tilde{x}_n)<\frac{1}{2^{k+1}},$$
  for all integers $m,n\geq n_k$ and for every integer $k\geq 1$.
  We define a sequence of elements $\{y_n\}$ in $\mathbb{X}$ inductively as follows:
  \begin{enumerate}[-]
  \item Let $y_1$ be an arbitrary element of $\tilde{x}_{n_1}.$
  \item For $k\geq 1$ let $y_{k+1}$ be an element of $\tilde{x}_{n_{k+1}}$ such that
    $d(y_k,y_{k+1})<\frac{1}{2^{k+1}}.$ Such an $y_{k+1}$ exists by
    Condition~(\ref{eq:condition-tilded}).
  \end{enumerate}
  Then the sequence $\{y_n\}$ is a Cauchy sequence in $(\mathbb{X},d)$. Indeed, for every $k\geq 1$
  and for all $m,n\geq k$ we have that
  $$
  d(y_m,y_n) \leq \sum_{l=m\wedge n}^{m\vee n -1} d(y_l,y_{l+1}) <
  \sum_{l=k}^{\infty}\frac{1}{2^{l+1}}=\frac{1}{2^k} \; .
  $$
  Since $(\mathbb{X},d)$ is complete, the sequence $\{y_n\}$ converges to some $x\in \mathbb{X}$.
  Let~$\tilde{x}$ be the equivalence class of $x$.  It follows that the sequence
  $\{\tilde{x}_{n_k}\}$ converges to $\tilde{x}$ because $\dtilde(\tilde{x}_{n_k},\tilde{x})\leq
  d(y_k,x)$ by definition od $\dtilde$. Finally, since $\{\tilde{x}_n\}$ is a Cauchy sequence, it
  follows easily that the whole sequence $\{\tilde{x}_n\}$ also converges to $\tilde{x}$, hence
  $(\tilde{\mathbb{X}},\dtilde)$ is complete.
\end{proof}

\begin{proof}[Proof of \Cref{lem:lo-complete}]
  Since we have $\|\theta^k\bx-\theta^l\by\|_\infty=\|\theta^{k-l}\bx-\by\|_\infty$ for all
  $\bx,\by\in l_0$ and $k,l\in\ZZ$, it follows that
  \begin{align*}
    \tilde{d}(\tilde{\bx},\tilde{\by}) =
    \inf\{\|\theta^k\bx-\by\|_\infty:k\in\ZZ\}=\inf\{\|\bx'-\by\|_\infty:\bx'\in\tilde{\bx}\} \; ,
  \end{align*}
  for all $\tilde{\bx},\tilde{\by}\in\lo$, and all $\bx\in\tilde{\bx},\by\in\tilde{\by}$. In view of
  Lemma \ref{lem:quotient-complet} it only remains to show that $\tilde{d}$ is a metric, rather than
  just a pseudo-metric.

  Assume that $\tilde{d}(\tilde{\bx},\tilde{\by})=0$ for some $\tilde{\bx},\tilde{\by}\in\lo$. Then,
  for arbitrary $\bx\in\tilde{\bx},\by\in\tilde{\by}$, there exists a sequence of integers $\{k_n\}$
  such that $\|\theta^{k_n}\bx-\by\|_\infty\rightarrow 0$, as $n\rightarrow\infty$. It suffices to
  show that the sequence $\{k_n\}$ is bounded. Indeed, by passing to a convergent subsequence it
  follows that there exists an integer $k$ such that $\by=\theta^k\bx$, hence
  $\tilde{\bx}=\tilde{\by}.$ Suppose now that the sequence $\{k_n\}$ is unbounded and that $\by\neq
  0$ (the case when $\by=0$ is trivial).  Without loss of generality, we can assume that
  $k_n\rightarrow\infty$, as $n\rightarrow \infty$. Since $\by\neq 0$ and $\lim_{|i|\to\infty}y_i=
  0$, there exists integers $i_0$ and $N>0$ such that $|y_{i_0}|=\|\by\|_\infty>0$ and
  $|y_{i}|<\|\by\|_\infty / 4$ for $|i|\geq N$. Since $\|\theta^{k_n}\bx-\by\|_\infty\rightarrow 0$
  there exists an integer $n_0>0$ such that $\|\theta^{k_n}\bx-\by\|_\infty<\|\by\|_\infty / 4$ for
  $n\geq n_0$. By our assumption, we can find an integer $n\geq n_0$ such that $k_n-k_{n_0}+i_0\geq
  N$, and it follows that
  \begin{align*}
    \frac{3}{4}\|\by\|_\infty < |(\theta^{k_n}\bx)_{i_0}| = |x_{k_n+i_0}| =
    |(\theta^{k_{n_0}}\bx)_{k_n-k_{n_0}+i_0}| < \frac{1}{2}\|\by\|_\infty \; , 
  \end{align*}
  which is a contradiction. Hence, the sequence $\{k_n\}$ is bounded.
\end{proof}

\subsection{\Cref{hypo:Aprimecluster} is a consequence of $\beta$-mixing}
\label{sec:betamixing}
The $\beta$-mixing coefficients of the sequence $\{X_j,j\in\ZZ\}$ are defined by 
\begin{align*}
  \beta_n = \frac12 \sup_{\mca,\mcb}  \sum_{i\in I} \sum_{j\in J}  | \PP(A_i\cap B_j) - \PP(A_i)\PP(B_j)| \; , 
\end{align*}
where the supremum is taken over all finite partitions $\mca=\{A_i,i\in I\}$ and
$\mcb=\{B_j,j\in J\}$ such that the sets $A_i$ are measurable with respect to $\sigma(X_k,k\leq0)$
and the sets $B_j$ are measurable with respect to $\sigma(X_k,k\geq n)$. See \cite[Section~1.6]{rio:2000}.
\begin{lemma}
  \label{lem:beta-implies-A}
  Assume that the sequence $\{X_j\}$ is $\beta$-mixing with coefficients
  $\{\beta_j,j\in\NN\}$. Assume that there exists a sequence $r_n$ satisfying \Cref{hypo:AC}
  and a sequence $\ell_n$ such that 
  \begin{align}
    \lim_{n\to\infty} \frac{\ell_n}{r_n} = \lim_{n\to\infty} \frac{n}{r_n} \beta_{\ell_n} = 0 \;
    .  \label{eq:condition-beta}
  \end{align}
  Then \Cref{hypo:Aprimecluster} holds.
\end{lemma}

\begin{proof}
  Write $X_{i,j}$ for $(X_i,\dots,X_j)$.  Set $k_n = \lfloor n/r_n\rfloor$ and let
  $\tilde{\mathbf{X}}_n$ be the vector of length $k_n(r_n-\ell_n)$ which concatenates all the
  subvectors $X_{(j-1)r_n+1,jr_n-\ell_n}$, $j=1,\dots,k_n$.  Let $\tilde{\mathbf{X}}_n^*$ be the
  vector build with independent blocks $X_{(j-1)r_n+1,jr_n-\ell_n}^*$ which each have the same
  distribution has the corresponding original blocks $X_{(j-1)r_n+1,jr_n-\ell_n}$.  Applying
  \cite[Lemma~2]{eberlein:1984} and~(\ref{eq:condition-beta}), we obtain
  \begin{align}
    \mathrm{d}_{TV}(\mathcal{L}(\tilde{\mathbf{X}}_n),\mathcal{L}(\tilde{\mathbf{X}}_n^*)) \leq k_n \beta_{\ell_n} =
    o(1) \; .  \label{eq:bound-TV}
  \end{align}
  Set $\tilde{X}_{n,i}= a_n^{-1}X_{(j-1)r_n+1,jr_n-\ell_n}$ and
  $\tilde{X}^*_{n,i}=a_n^{-1}X^*_{(j-1)r_n+1,jr_n-\ell_n}$ and define the following point processes
  \begin{align*}
    \tilde{N}''_n = \sum_{i=1}^{k_n} \delta_{(i/k_n,\tilde{X}_{n,i})} \; , \ \ \tilde{N}^*_n =
    \sum_{i=1}^{k_n} \delta_{(i/k_n,X^*_{n,i})} \; .
  \end{align*}
  Let $f$ be a nonnegative function defined on $[0,1]\times\tilde\ell_0$. 
  Since the exponential of a negative function is less than~1, by definition of the total variation
  distance, the bound (\ref{eq:bound-TV}) yields
  \begin{align}
    \left| \EE\left[ \rme^{-\tilde{N}''_n(f)}\right] - \EE\left[ \rme^{-\tilde{N}^*_n(f)}\right]\right| \leq
      \mathrm{d}_{TV}(\mathcal{L}(\tilde{\mathbf{X}}_n),\mathcal{L}(\tilde{\mathbf{X}}_n^*)) = o(1) \;
      .  \label{eq:bound-laplace}
  \end{align}
  We must now check that the same limit holds with the full blocks instead of the truncated blocks.
  Under \Cref{hypo:AC} (which holds for any sequence smaller than $r_n$ hence for $\ell_n$), we know
  by \cite[Proposition~4.2]{basrak:segers:2009} that for every $\epsilon>0$ and every sequence
  $\{\ell_n\}$ such that $\ell_n\to\infty$ and $\ell_n\bar{F}(a_n)\to0$,
  \begin{align}
    \lim_{n\to\infty} \frac{\PP(\max_{1\leq i \leq \ell_n} |X_i| >\epsilon a_n) }{\ell_n \PP(|X_0|>a_n)} 
    = \theta \epsilon^{-\alpha} \; .  \label{eq:bigO}
  \end{align}
  Then, applying (\ref{eq:bigO}) yields,
  \begin{align*}
    \PP\left(\max_{1\leq j\leq k_n} \max_{1\leq i \leq \ell_n} |X_{jr_n-i+1}|>\epsilon a_n\right)
    &    \leq k_n \PP\left( \max_{1\leq i \leq \ell_n} |X_{i}|>\epsilon a_n\right) \\
    & = O(k_n \ell_n\PP(|X_0|>a_n)) = O(\ell_n/r_n) = o(1) \; .
  \end{align*}
  Assume now that $f$ depends only on the components greater than $\epsilon$ in absolute value.
  Then $\PPC_n(f)=\tilde{N}''_n(f)$ unless at least one component at the end of one block is greater
  than~$\epsilon$. This yields
  \begin{align*}
    \left| \EE[\rme^{-N_n''(f)}] -  \EE[\rme^{-\tilde{N}_n''(f)}] \right| 
      \leq \PP\left( \max_{1\leq j\leq k_n} \max_{1\leq i \leq \ell_n} |X_{jr_n-i+1}|>\epsilon a_n\right) = o(1) \; .
  \end{align*}
  The same relation also holds for the independent blocks. Therefore, \Cref{hypo:Aprimecluster}
  holds. %
\end{proof}

\subsection{On continuity of addition in $E$}
The next lemma gives sufficient conditions for continuity of addition in the space $E([0,1],\RR)$.

\begin{lemma}
  \label{lem:ContinuityOfAdditionInE}
  Suppose that $\sequence{x}[\NN][n]$ is a sequence in $D([0,1],\RR)$ and $x'=(x,S,\{I(t) : t\in S\})$
  an element in $E$ such that $x_n\rightarrow x'$ in $E.$ Suppose also that $\sequence{b}[\NN][n]$
  is a sequence in $D([0,1],\RR)$ which converges uniformly to a continuous function $b$ on $[0,1].$
  Then the sequence $\{x_n-b_n\}$ converges in $(E,\metricE)$ to an element $x'-b\in E$ defined by
  $$
  x'-b = (x-b,S,\{I(t)-b(t): t\in S\}) \; .
  $$
\end{lemma}

\begin{proof}
  Recall the definiton of $\metricE$ given in \eqref{eq:HausE}.  By Whitt~\cite[Theorem
  15.5.1.]{whitt:2002} to show that $x_n-b_n \rightarrow x'-b$ in $E,$ it suffices to prove that
  \begin{equation}
    \label{eq:GraphConvergence}
    \sup_{(t,z)\in \Gamma_{x_n-b_n}}\|(t,z)-\Gamma_{x'-b}\|_{\infty}\rightarrow 0.
  \end{equation}
Take an arbitrary $\epsilon>0.$ Note that $b$ is uniformly continuous so by the conditions of the
lemma there exists $0<\delta\leq \epsilon$ and $n_0\in\NN$ such that
\begin{itemize}
\item[(i)] $|t-s|<\delta \Rightarrow |b(t)-b(s)|<\epsilon,$
\item[(ii)] $\metricE(x_n,x')<\delta,$ for all $n\geq n_0$ and
\item[(iii)] $|b_n(t)-b(t)|<\epsilon,$ for all $t\in[0,1].$
\end{itemize}
Also, since $b$ is continuous, it easily follows that $|b_n(t)-b_n(t-)|\leq 2\epsilon$ for all $n\geq n_0$ and $t\in[0,1].$ 

Take $n\geq n_0$ and a point $(t,z)\in \Gamma_{x_n-b_n},$ i.e. 
\[ 
z\in [(x_n(t-)-b_n(t-))\wedge (x_n(t)-b_n(t)),(x_n(t-)-b_n(t-))\vee (x_n(t)-b_n(t))].
\] 
Since $|b_n(t)-b_n(t-)|\leq 2\epsilon$ there exists $z'\in [x_n(t-)\wedge x_n(t),x_n(t-)\vee x_n(t)]$ (i.e. $(t,z')\in \Gamma_{x_n}$), such that
\[
|(z'-b_n(t))-z|\leq 2\epsilon.
\]
Next, since $\metricE(x_n,x')<\delta,$ there exists a point $(s,y)\in \Gamma_{x'}$ such that
\[
|s-t|\vee|y-z'|<\delta.
\]
Note that $(s,y-b(s))\in \Gamma_{x'-b}$ and by previous arguments
\begin{align*}
  |(y-b(s))-z|&=|(y-b(s))-z + (z' -b_n(t)) - (z' -b_n(t)) + b(t) - b(t)|\\
              &\leq |y-z'|+|b(t)-b(s)|+|b_n(t)-b(t)|+|(z'-b_n(t))-z|\\
              &\leq \delta + \epsilon + \epsilon + 2\epsilon\\
              &\leq5\epsilon.
\end{align*}
Also, $|s-t|<\delta\leq \epsilon.$ Hence, for all $n\geq n_0,$
\[
\sup_{(t,z)\in \Gamma_{x_n-b_n}}\|(t,z)-\Gamma_{x'-b}\|_{\infty}\leq 5\epsilon.
\]
and since $\epsilon$ was arbitrary, (\ref{eq:GraphConvergence}) holds.
\end{proof}

\subsection{A lemma for partial sum convergence in $E$}

\begin{lemma}
  \label{lem:newlem}
  Let $\alpha\in[1,2)$ and let the assumptions of \Cref{thm:PartialSumConvinE} hold.  Then there
  exists an $\alpha$-stable L\'evy process $V$ on $[0,1]$ such that, as $\epsilon\rightarrow 0,$ the
  process $V_\epsilon$ defined in (\ref{eq:Vepsilon}) converges uniformly a.s.  (along some
  subsequence) to $V$.
\end{lemma}

\begin{proof}
  Recall that
  \begin{align*}
    V_\epsilon(t) & = \sum_{T_i\leq t} s^\epsilon(P_i \bsQ_i) - t \int_{\{x\; : \;\epsilon < |x| \leq 1\}} x \mu(dx) \; 
  \end{align*}
  where 
\begin{equation*}
  \mu(dx)=p\alpha x^{-\alpha-1} \1{(0,\infty)}(x) \rmd x + (1-p)\alpha (-x)^{-\alpha-1} \1{(-\infty,0)}(x) \rmd x
\end{equation*}
for $p=\PP(\Theta_0=1)$. We first show that the centering term can be
  expressed as an expectation of a functional of the limiting point process $\PPC$. More precisely,
  we show that for all $\epsilon>0$
  \begin{equation}
    \label{eq:equality_of_centerings}
    \int_{\{x\; : \;\epsilon < |x| \leq 1\}} x \mu(dx) =\theta \int_0^\infty \EE\left[y\sum_{j\in\ZZ} Q_j \1{\{\epsilon < y |Q_j| \leq 1\}} \right] 
    \alpha y^{-\alpha -1} \rmd y \; .
  \end{equation}
First, as shown in \cite[Theorem 3.2, Equation (3.13)]{davis:hsing:1995} it holds that
  \begin{align}\label{eq:DH95identity}
  \theta \EE\left[\sum_{j\in \ZZ} Q_j |Q_j|^{\alpha -1}\right]=2p-1
  \end{align}
  so by Fubini's theorem, if $\alpha>1$
  \begin{align*}
    \theta \int_0^\infty \EE\left[y\sum_{j\in\ZZ} Q_j \1{\{\epsilon < y |Q_j| \leq 1\}} \right] \alpha y^{-\alpha -1} \rmd y 
    & = \alpha \theta \EE\left[\sum_{j\in\ZZ} Q_j \int_{\epsilon |Q_j|^{-1}}^{|Q_j|^{-1}}  y^{-\alpha} \rmd y  \right] \\
    & = \frac{\alpha}{\alpha - 1} (\epsilon^{-\alpha +1 } - 1) \theta \EE\left[\sum_{j\in \ZZ} Q_j |Q_j|^{\alpha -1}\right]\\
    & = \frac{\alpha}{\alpha - 1} (\epsilon^{-\alpha +1 } - 1) (2p-1)  \;,
  \end{align*}
  and if $\alpha=1$ the same term equals $ \log(\epsilon^{-1}) (2p-1)$. Note that the use of
  Fubini's theorem is justified since the same calculation as above shows that the above integral
  converges absolutely since $\EE[\sum_{j\in \ZZ} |Q_j|^{\alpha}] < \infty$. The equality in
  \eqref{eq:equality_of_centerings} now follows by the definition of the measure $\mu$. Hence, for
  all $t\in [0,1]$
  \begin{align*}
    V_\epsilon(t)   = \sum_{T_i\leq t} s^\epsilon(P_i \bsQ_i) 
    - t \theta \int_0^\infty  \EE\left[y\sum_{j\in\ZZ} Q_j \1{\{\epsilon < y |Q_j| \leq 1\}} \right] \alpha y^{-\alpha -1} \rmd y \; .  
  \end{align*}  
  Recall from \Cref{rem:useful-W} that we can define $W=\sum_{j\in\ZZ} |Q_j|$,
  $W_i=\sum_{j\in \ZZ}|Q_{i,j}|$ so that $\{W_i,i\geq 1\}$ is a sequence of \iid\ random variables
  with the same distribution as $W$ and $\EE[W_i^\alpha] <\infty$ and that
  $\sum_{i=1}^\infty\delta_{P_i W_i}$ is a Poisson point process on $(0,\infty]$ with intensity
  measure $\theta \EE [W_1^\alpha] \alpha y^{-\alpha -1}dy.$ In particular, for every $\delta>0$
  there are almost surely at most finitely many points $P_i W_i$ such that $P_i W_i > \delta$. For
  $\delta,\epsilon>0$, define
  \begin{align*}
    m_{\epsilon,\delta} & = \theta \int_0^\infty \EE\left[y\sum_{j\in\ZZ} Q_j \1{\{\epsilon <
                          y|Q_j|\leq 1,\;\delta < yW\}}\right]\alpha y^{-\alpha -1}dy \; .
  \end{align*}
  Note that $\lim_{\epsilon\to0} m_{\epsilon,\delta}=m_{0,\delta}$ for all $\delta>0$ by the
  dominated convergence theorem.  Indeed, if $\alpha=1$ we have that
  \begin{multline*}
    \theta \int_0^\infty \EE\left[y\sum_{j\in\ZZ} |Q_j| \1{\{y|Q_j|\leq 1,\;\delta < yW\}} \right]
    \alpha y^{-\alpha -1} \rmd y \leq \theta \EE \left[ \sum_{j\in\ZZ} |Q_j|
      \int_{\frac{\delta\wedge1}{W}}^{\frac{1}{|Q_j|}} y^{-1} \rmd y \right]    \\
    = \theta \EE\left[ \sum_{j\in\ZZ} |Q_j| \log(|Q_j|^{-1}) + W \log W + \log((\delta\wedge1)^{-1})  W \right] \; ,
  \end{multline*}
  which is finite by assumption \eqref{eq:SumOfTheQjs_alpha_equal_to_1}, and if $\alpha>1$ similar
  calculation using the assumption $\EE[W^\alpha] <\infty$ justifies the use of the dominated
  convergence theorem.  

  Since for every $\delta>0$ there a.s. exists at most finitely many points $P_i W_i$ such that
  $P_i W_i >\delta$, for every $\epsilon \geq 0$ we can define the process $V_{\epsilon,\delta}$ in
  $D[0,1]$ by
  \[
  V_{\epsilon,\delta}(t)=\sum_{T_i\leq t} s^\epsilon(P_i \bsQ_i)\1{\{\delta < P_iW_i\}} -
  tm_{\epsilon,\delta}=\sum_{T_i\leq t}\sum_{j\in\ZZ}P_i Q_{i,j} \1{\{\epsilon<P_i |Q_{i,j}|,\;
    \delta < P_iW_i\}} - t m_{\epsilon,\delta} \; .
  \]
  Furthermore, for every fixed $\delta>0$, as $\epsilon \to 0$, $V_{\epsilon,\delta}$ converges
  uniformly almost surely to $V_{0,\delta}$.

  Next, we prove that for any positive sequence $\{\delta_k\}$ with $\delta_k\searrow 0$ as
  $k\to \infty$, $V_{0,\delta_k}$ converges uniformly almost surely to a process $V$ in $D([0,1])$.
  Note first that by \cite[Theorem 3.1]{davis:hsing:1995} the finite dimensional distributions of $V_{0,\delta}$ converge to those of an
  $\alpha$-stable L\'evy process.

  Since $\sum_{i\geq1}\delta_{T_i,P_i,\bsQ_i}$ is a Poisson point process on
  $[0,1]\times(0,\infty]\times\lo$, the process $V_{0,\delta}$ has independent increments with
  respect to $\delta$, that is for every $\delta<\delta'$, $V_{0,\delta}-V_{0,\delta'}$ is
  independent of $V_{0,\delta'}$.  Moreover, since $V_{0,\delta}-V_{0,\delta'}$ is a Poisson
  integral, we have that
  \begin{align*}
    \var(V_{0,\delta}(1)-V_{0,\delta'}(1) )
    & = \theta \int_0^\infty y^2 \EE \left[ \left( \sum_{j\in\ZZ} Q_j \right)^2\1{\{\delta  < y W \leq \delta'\}}  \right] \alpha y^{-\alpha-1} \rmd y  \\
    & \leq \theta \EE \left[ W^2 \int_0^{\delta'/W} \alpha y^{-\alpha+1} \rmd y \right] = \frac{\theta\alpha
      (\delta')^{2-\alpha}}{(2-\alpha)} \EE[W^\alpha] \; .
  \end{align*}
  Therefore, $\lim_{\delta' \to 0} \var(V_{0,\delta}(1)-V_{0,\delta'}(1) )=0$ and now arguing
  exactly as in the proof of \cite[Proposition~5.7, Property~2]{resnick:2007} shows that for any
  positive sequence $\{\delta_k\}$ with $\delta_k\searrow 0$, $\{V_{0,\delta_k}\}$ is almost surely
  a Cauchy sequence in $D([0,1])$ with respect to the supremum metric $\|\cdot\|_\infty$. Since the
  space $D([0,1])$ is complete under this metric, we obtain the existence of the process
  $V=\{V(t), t\in [0,1]\}$ with paths in $D([0,1])$ almost surely and such that
  $\lim_{k\to\infty}\|V_{0,\delta_k}-V\|_\infty =0$ almost surely.

  There only remains to prove that for all $u>0$,
  \begin{align}
    \label{eq:slutsky}
    \lim_{\delta\to0} \limsup_{\epsilon\to0} \PP( \|V_\epsilon-V_{\epsilon,\delta}\|_\infty > u) = 0 \; .
  \end{align}
  Indeed, this would imply that $\|V_\epsilon-V\|_\infty\to 0$ in probability and hence that, along
  some subsequence, $V_\epsilon$ converges to $V$ uniformly almost surely.  Since for
  $\delta\leq 1$, $yW=\sum_{j\in\ZZ}y|Q_j|\leq \delta$ implies that $y|Q_j|\leq \delta \leq 1$ for
  all $j\in\ZZ$, we have that
  \begin{multline*}
    V_\epsilon(t)-V_{\epsilon,\delta}(t)
    = \sum_{T_i\leq t}\sum_{j\in\ZZ}P_i Q_{i,j} \1{\{\epsilon < P_i|Q_{i,j}|,\; P_iW_i \leq \delta \}} \\
    - t \theta \int_0^\infty \EE \left[ y\sum_{j\in\ZZ} Q_j \1{\{\epsilon < y|Q_j|,\; y W \leq
        \delta\}} \right] \alpha y^{-\alpha -1} \rmd y \; .
  \end{multline*}
  The process $V_\epsilon-V_{\epsilon,\delta}$ is a c\`adl\`ag martingale, thus applying
  Doob-Meyer's inequality yields
  \begin{align*}
    \PP\left( \|V_\epsilon-V_{\epsilon,\delta}\|_\infty > u \right) 
    & \leq u^{-2} \var(V_\epsilon(1)-V_{\epsilon,\delta}(1) ) \\
    & = u^{-2} \theta\int_0^\infty y^2 \EE \left[ \left( \sum_{j\in\ZZ} Q_j \1{\{\epsilon < y|Q_j|,\; y W \leq \delta\}} \right)^2 \right] 
      \alpha y^{-\alpha-1} \rmd y  \\
    & \leq u^{-2} \theta\EE \left[ W^2 \int_0^{\delta/W} \alpha y^{-\alpha+1} \rmd y \right] = \frac{\theta\alpha
      \delta^{2-\alpha}}{u^2(2-\alpha)} \EE[W^\alpha] \; 
  \end{align*}
  and hence $\eqref{eq:slutsky}$ holds.
\end{proof}

\subsection{The parameters of the $1$-stable random variable $V(1)$}

\begin{lemma}\label{lem:alpha_stable_parameters} 
In the case $\alpha=1$, the characteristic function of $V(1)$ where $V$ is the $1$-stable L\'evy process from \Cref{thm:PartialSumConvinE} is given by
  \begin{align}
    \label{eq:charac-V}
    \log \EE[\rme^{\rmi z V(1)}] = 
      \rmi \location z -
      \frac\pi2 \sigma|z|\{1 - \rmi \frac2\pi\skewness \mathrm{sgn}(z) \log(|z|)\}    
\end{align}
with 
\begin{align*}
  \sigma & = \theta\EE\left[\left|\sum_{j\in \ZZ}Q_j\right|\right] \; , \ \
  \skewness  = \frac{\EE[(\sum_{j\in \ZZ}Q_j)_+]-\EE[(\sum_{j\in \ZZ}Q_j)_-]}{\EE[|\sum_{j\in \ZZ}Q_j|]}
\end{align*}
and 
\[
\location = \theta\left(c_0 \EE\left[\sum_{j\in\ZZ} Q_j\right] -  \EE\left[\sum_{j\in\ZZ} Q_j
  \log\left(\left| \sum_{j\in\ZZ} Q_j \right|\right)\right]-\EE\left[\sum_{j\in\ZZ} Q_j
  \log\left(\left|Q_{j} \right|^{-1}\right)\right]\right) \; ,
\]
with $c_0 = \int_{0}^{\infty}(\sin y - y\1{(0,1]}(y))y^{-2} \rmd y$.
\end{lemma}

\begin{proof}
As shown in the proof of \Cref{lem:newlem}, $V(1)$ is the distributional limit
  of the sequence of random variables $\{V_{0,\delta_k}(1), k\in\NN\}$ for any positive sequence
  $\{\delta_k\}$ such that $\delta_k\searrow 0$, where for $\delta>0$
\begin{align*}
V_{0,\delta}(1)&=\sum_{i\in \NN}\sum_{j\in\ZZ}P_i Q_{i,j} \1{\{\delta < P_iW_i\}} - 
  \theta \int_0^\infty \EE\left[y\sum_{j\in\ZZ} Q_j \1{\{y|Q_j|\leq 1,\;\delta < yW\}}\right]\alpha y^{-\alpha -1}dy 
\end{align*}
where $\sum_{i\in \NN}\delta_{T_i,P_i,\{Q_{i,j}\}_{j\in\ZZ}}$ is a Poisson point process on
$[0,1]\times(0,\infty]\times\lo$ with intensity measure $Leb\times d(-\theta y^{-\alpha})\times \PP_{\bsQ} $ and $W=\sum_{j\in\ZZ} |Q_j|$,
$W_i=\sum_{j\in \ZZ}|Q_{i,j}|$ . Hence for all $z\in\RR$
\[
\log\EE\left[e^{\rmi zV(1)} \right]=\lim_{k\to \infty}\log\EE\left[e^{\rmi zV_{0,\delta_k}(1)} \right] \; .
\]
 Since $yW\leq 1$ implies that $y|Q_j|\leq 1$ for all $j\in\ZZ$, for all $\delta<1$ we have that
\begin{multline*}
  V_{0,\delta}=\left(\sum_{i,j} P_i Q_{i,j}\ind{ \delta < P_i W_i}-\theta \int_0^\infty
    \EE\left[y\sum_{j\in\ZZ} Q_j \1{\{\delta < yW \leq 1\}}\right] y^{-2}dy\right)  \\
  - \theta \int_0^\infty \EE\left[y\sum_{j\in\ZZ} Q_j \ind{ y|Q_j|\leq 1,\;1< yW}\right] y^{-2}dy \;
  .
\end{multline*}
By Fubini's theorem, the last term above is equal to 
\begin{align*}
  \theta \EE\left[ \sum_{j\in\ZZ} Q_j \log(W) \right] + \theta \EE\left[ \sum_{j\in\ZZ} Q_j \log(|Q_j|^{-1}) \right] \;  ,
\end{align*}
(with the usual convention $0\log 0=0$).  Therefore, for all $z\in\RR$ and $\delta<1$
\begin{multline}
  \label{eq:characteristic}
  \log\EE\left[e^{\rmi zV_{0,\delta}(1)} \right]=\theta\int_0^\infty
  \EE\left[ \left\{e^{\rmi zyS}-1-\rmi zyS\ind{yW \leq 1}\right\}\ind{\delta < yW} \right] y^{-2}dy \\
  - \rmi z\theta \EE\left[ \sum_{j\in\ZZ}Q_j \log(|Q_j|^{-1} W) \right] \; .
\end{multline}
where $S=\sum_{j\in\ZZ} Q_j$. Since for all $\delta<1$, using the fact that
$| e^{\rmi z}-1-\rmi z|\leq |z|^2 / 2 \leq |z|^2 $ for all $z\in\RR$ (see for example \cite[Lemma
8.6]{sato:1999}) and $\EE [W]<\infty$,
\begin{multline*}
  \EE\left[\int_0^\infty \left| e^{\rmi zyS}-1-\rmi zyS\ind{yW \leq 1}\right|\ind{\delta < yW}
    y^{-2}dy\right] \\ \leq \EE\left[\int_{1/W}^\infty \left| e^{\rmi zyS}-1\right| y^{-2}dy\right]
  +\EE\left[\int_{\delta / W}^{1/W}  \left| e^{\rmi zyS}-1-\rmi zyS\right|  y^{-2}dy\right] \\
  \leq 2\EE[ W] + |z|^2(1-\delta) \EE [W]\leq (2+|z|^2)\EE [W]<\infty \; ,
\end{multline*}
by the dominated convergence theorem, as $\delta\to 0$  the first term on the right side of  \eqref{eq:characteristic} tends to
\begin{multline*}
  \theta\int_0^\infty \EE\left[ \left\{e^{\rmi zyS}-1-\rmi zyS\ind{yW \leq 1}\right\} \right] y^{-2}dy \\
  =\theta\EE\left[\int_0^\infty \left\{e^{-\rmi zyS}-1-\rmi zyS\1{(0,1]}(y)\right\} y^{-2}dy\right]
  + \rmi z\theta \EE\left[S\int_{1/W}^1 y^{-1}dy \right]\; .
\end{multline*}
Altogether, using the integral from \cite[Page 85]{sato:1999} we get that
\begin{multline*}
  \lim_{\delta\to 0}\log\EE\left[e^{\rmi zV_{0,\delta}(1)}
  \right]=-\theta\frac{\pi}{2}|z|\EE[|S|]-i\theta z\log|z|\EE[S]
  -\rmi z\theta\EE[S\log|S|]+ \rmi c_0\theta z\EE[S]\\
  +\rmi z\theta\EE[S\log W]-\rmi z\theta\EE[S\log W]-\theta \EE[\sum_{j\in\ZZ}Q_j \log(|Q_j|^{-1}) ]  \; ,
\end{multline*}
where 
\[
c_0 = \int_{0}^{\infty}\frac{\sin y - y\1{(0,1]}(y) }{y^2} dy \; .
\]
Setting $\sigma=\theta\EE[|S|]$, $\skewness=\frac{\EE[S]}{\EE[|S|]}$ and 
\[
\location=\theta\left( c_0\EE[S] -\EE[S\log|S|] -\EE\left[\sum_{j\in\ZZ}Q_j \log(|Q_j|^{-1}) \right] \right) \; ,
\]
since the term $\rmi z\theta\EE[S\log W]$ cancels out, we obtain that
\begin{align*}
  \log\EE\left[e^{\rmi zV(1)} \right]
  & = \lim_{\delta\to 0}\log\EE\left[e^{\rmi zV_{0,\delta}(1)} \right] \\
  & = - \frac{\pi}{2}\sigma |z|\left(1+\rmi \frac{2}{\pi} \skewness \text{sgn}(z) \log |z|   \right) + \rmi\location z \; .
\end{align*}

\end{proof}

\section*{Acknowledgements}  
Parts of this paper were written when Bojan Basrak visited the Laboratoire MODAL'X at Universit\'e
Paris Nanterre. Bojan Basrak takes pleasures in thanking MODAL'X and for excellent hospitality
and financial support, as well as Johan Segers for useful discussions over the years.  The work of
Bojan Basrak and Hrvoje Planini\'c has been supported in part by Croatian Science Foundation under
the project 3526.  The work of Philippe Soulier was partially supported by LABEX MME-DII.

\bibliographystyle{apalike}

\end{document}